\documentclass[article,onefignum,onetabnum]{siamart190516}
\usepackage{xcolor}
\usepackage{bbold}
\usepackage{soul}
\usepackage{comment}
\usepackage[cal=boondox,scr=boondoxo]{mathalfa}
\usepackage{amsmath,amssymb}
\DeclareMathAlphabet{\mathpzc}{OT1}{pzc}{m}{it}
\def \B {{\mathfrak  B}}
\def \U {{\mathcal  U}}
\def \CC{{\mathbf C}}
 \def\bel{\begin{equation}\label}
\def\eeq{\end{equation}}
\def\cS{\mathfrak{T}}

\def\h{\Psi}

\def \va{\mathpzc{v}}

\def\vz{\mathpzc{v}}

\DeclareMathAlphabet{\mathpzc}{OT1}{pzc}{m}{it}

\newcommand{\cR}{\mathbb{R}}
\newcommand{\cN}{\mathbb{N}}

\def\ds{\displaystyle}

\def\xbo{\check x}

\def\A{\mathcal{A}}
\def\B{\mathfrak{B}}
\def\C{\mathcal{C}}

\def\N{\mathcal{N}}

\def\T{\mathcal{T}}
\def\U{\mathcal{U}}

\def\W{\mathcal{W}}

\def\xbo{\check x}

\def\d{ {\rm d} }
\def\vsmm{\vskip0.2truecm}

\def\begi{\begin{itemize}}
	\def\endi{\end{itemize}}

\def\ds{\displaystyle}

\def\N{{\cal M}}

\def\W{{\mathcal W}}

\def\K{{\mathcal K}}
\def\C{{\cal C}}
\def\T{{\cruk}}

\def\rr{\mathbb{R}}

\def\sqr#1#2{\vbox{\hrule height .#2pt
		\hbox{\vrule width .#2pt height #1pt \kern #1pt
			\vrule width .#2pt}\hrule height .#2pt }}

\def\bega{\begin{array}}
	\def\enda{\end{array}}
\def\begi{\begin{itemize}}
	\def\endi{\end{itemize}}

\def\vz{\mathpzc{v}}
\def\cruk{{\mathfrak T}}

\renewcommand{\thesection}{\arabic{section}}
\newtheorem{thm}{Theorem}[section]
\usepackage{yfonts}

 \newtheorem{lma}{Lemma}[section]
\newtheorem{prop}{Proposition}[section]
\newtheorem{remark}{Remark}[section]

\usepackage{lipsum}
\usepackage{amsfonts}
\usepackage{graphicx}
\usepackage{epstopdf}
\usepackage{algorithmic}

\headers{Strict minimizers and higher-order abnormality}{M. Motta and M. Palladino and F. Rampazzo}

\title{Minimizers that are not impulsive minimizers and higher order abnormality 
\thanks{Submitted to the editors DATE.
}}

\author{Monica Motta\thanks{Department of Mathematics ``Tullio Levi-Civita", University of Padova, Italy 
  (\email{fmotta@math.unipd.it}).}
\and Michele Palladino\thanks{Department of Engineering, Information Sciences and Mathematics, University of  L'Aquila, Italy 
  (\email{michele.palladino@univaq.it}).}
\and Franco Rampazzo\thanks{Department of Mathematics ``Tullio Levi-Civita", University of Padova, Italy 
  (\email{rampazzo@math.unipd.it}).}
}

\usepackage{amsopn}

\begin{document}




\maketitle

\begin{abstract}
	This paper addresses two related problems in optimal control.  The first investigation consists of compatibility issues between two classical approaches to deriving necessary conditions for  optimal control problems with a  final target: the {\it set-separation approach} and {\it penalization techniques.} These methods generally lead to non-equivalent conditions, mainly due to their reliance on different notions of tangency at the target.
	We address this issue by considering {\it Quasi Differential Quotient (QDQ) approximating cones} (which are fit for the set-separation approach ) and identifying conditions under which the Clarke tangent cone (which is a typical tool within penalization techniques) is also a QDQ approximating cone. In particular, we show that this property holds under suitable local invariance assumptions or when the target coincides locally with an 
	{\it $r$-prox regular set}. 
	In the second part of the paper we apply this compatibility result to the study of {\it infimum-gap phenomena} in optimal control problems with unbounded controls and impulsive extensions. In particular, we establish a connection between the occurrence of infimum gaps for {\it strict-sense minimizers} and abnormality in a higher-order Maximum Principle involving Lie brackets. While the abnormality–gap correspondence beyond first-order conditions   has been already established for extended-sense --i.e. impulsive-- minimizers,  a topological argument involving the former and the utilization of the above compatibility issues allow us to extend this correspondence to  strict-sense minimizers.
\end{abstract}

  
\section{Introduction}

Non-equivalent approaches to deriving necessary conditions for a continuous-time optimal control problem with a closed final target $\mathcal{S}$ can be found in the literature. In particular, two methods have been developed: on the one hand, the so-called {\it set-separation approach} (see, e.g. \cite{SussCDC05}-\cite{Suss-ho}, \cite{warga2,warga}); on the other hand, 
	a penalization technique,  rooted in Ekeland's variational principle (see, e.g. \cite{vinter},\cite{ClarkeLed}). 
	 It is well known that  these approaches may yield non-equivalent  necessary conditions, mainly because they rely on different notions of {\it tangency} at   points of the  target (see \cite{B}). 
However, under suitable {\it compatibility} assumptions, nothing precludes studying some issues   connected with important properties of necessary conditions  using both approaches. This idea  motivates the first part of the present paper. In the second part,  we apply this compatibility framework  to investigate an infimum-gap problem. The main obstacle to combining the  set-separation approach with  the penalization technique lies in the fact that, in general, a tangent cone -- e.g. the Clarke tangent cone--
to a point $\bar x$ of the final target $\mathcal{S}$  is not what is called, in the set-separation approach, an  {\it approximating cone}. In this work we adopt the notion of a
 \textit{QDQ approximating cone}, which encompasses several types of approximating cones appearing in the literature.

We address this theoretical question in the next subsection, while in the subsequent  subsection we will discuss an application to a concrete problem, related to the {\it infimum-gap} question.

\subsection{A theoretical issue} 
  We first introduce the notion of an approximating cone. We use the acronym QDQ to mean  {\it Quasi Differential Quotient}.  We refer to Section \ref{S_T_QDQ} for the  definition of 
  QDQ, a notion of a generalized differential (which in our application will be an element of $Lin(\rr^N,\rr^n)$, however, in general it can be a subset of such linear maps) introduced in \cite{PR}. It coincides with a sort of regular instance of Sussmann's Approximate Generalized Differential Quotient. A QDQ {\it approximating cone $\K$    to  a set $\mathcal{S}\subseteq\rr^n$ at  a point $\bar x\in \mathcal{S}$} is  the image $\K =
  L\cdot\Gamma$, where  $\Gamma$ is   a cone (of dimension $N\leq n$) and $L\in Lin(\rr^N,\rr^n)$ is a QDQ   of a map from a subset of $\rr^N$ into $\mathcal{S}$. 
 
 We will also make use of the well-known {\it Clarke  tangent cone}. We recall  (see Section \ref{subGenCone}) that  the Clarke  tangent cone to $\mathcal{S}$ at $\bar x$, here denoted  by $T^C_\mathcal{S}(\bar x)$, can be defined 
 as follows:
$$\ds T_\mathcal{S}^C(\bar x):=\left\{v\in   \rr^n: \ \ \underset{t\downarrow 0, \, x\to \bar x}{\limsup}\,\frac{{d_\mathcal{S}(x+t v)-d_\mathcal{S}(x)}}{t}\right\}$$   (where, for any $z\in\rr^n$ and any  ${\mathcal A}\subseteq\rr^n$,  $d_{\mathcal A}(z)$ denotes the distance of the point  $z$ from the subset ${\mathcal A}$, that is, $d_{\mathcal A}(z)=\inf\{|z-x|:\ x\in \mathcal A\}$).
The afore-mentioned compatibility  issue can be expressed as follows:

{\bf Q1.}{\it  Under which assumptions on $\mathcal{S}$ does the Clarke  tangent cone $T_\mathcal{S}^C(\bar x)$    also constitute a  QDQ approximating cone   to $\mathcal{S}$ at $\bar x$} ?

To give an answer to question {\bf Q1}, we need the following  notion of  {\it $r$-prox regular set } (see Section \ref{pr_sect}), a notion that generalizes both the concepts of convex set and  of $C^2$ set. Roughly speaking, for   $r>0$, a subset $\A\subseteq  \rr^n$ is a {\it $r$-prox regular set} (also called a set {\it with positive reach}) 
if  for every $y\in  \rr^n\backslash \A$ such that $d_\A(y)\leq r$  there exists a unique projection of 
$y$ onto the boundary of $\A$ (denoted by $\partial \A$).

Our answer to {\bf Q1} is the following (see Theorem \ref{Th_T^C}):
 
{\em {\bf Result 1.}  If there exists a real number $\delta>0$   such that  either {\bf (i)} $(\bar x+T^C_\mathcal{S}(\bar x))\cap B_\delta(\bar x)\subseteq \mathcal{S}$, or {\bf (ii)} $\mathcal{S}\cap  B_\delta(\bar x)$ coincides with the restriction to $B_\delta(\bar x)$ of a $r$-prox regular set  for some $r>0$, then the Clarke tangent cone $T^C_\mathcal{S}(\bar x)$ to $\mathcal{S}$ at $\bar x$ is a QDQ approximating cone to $\mathcal{S}$ at $\bar x$.}
 
 As a general consequence, it follows that, under  hypothesis  {\bf (i)} or {\bf (ii)}  in {\bf Result 1},  any statement about necessary conditions  expressed in terms of QDQ approximating cones to a target set $\mathcal{S}$ (within the set-separation approach)  implies  analogous necessary conditions  in which the so-called transversality condition is formulated in terms of the Clarke normal cone. As we shall see in the application to the infimum-gap phenomenon (see the next subsection), this is particularly helpful when one is dealing with higher-order necessary conditions (involving Lie brackets of the dynamical vector fields).
 
 \subsection{An application to the question of infimum gaps}
 
 When the infimum of an optimal control problem is not attained, a standard strategy consists of extending the domain of optimization to its closure with respect to an appropriate topology. Classical examples include {\it relaxation} (by convexification or extension to probability measures)    or {\it impulsive extension} of non-coercive problems with unbounded controls. Clearly, an important prerequisite of such  extensions is that no  infimum gaps occur, which means that  the infima of the original and extended problems coincide. Furthermore, one  should be able to construct near-optimal processes for the original problem from a minimizer of the extended one. Therefore,  identifying sufficient conditions that  exclude an infimum gap is crucial both for the theoretical foundation and for the numerical implementation of optimal control methods.   Let us briefly explain how the goal of avoiding infimum gaps is connected with the theoretical issue described in the previous subsection.  
 

Since the early contributions of Warga \cite{warga2}--\cite{warga3}, it has been known that, in relaxed problems that exhibit an infimum gap, the extended minimizer  satisfies  necessary conditions in the form of  an {\it abnormal} Maximum Principle: this means that the multiplier connected with the cost is  equal to zero.  Within the penalization approach, this link between abnormality and infimum gap has been recognized also for  nonsmooth data and state-constrained problems, see e.g. \cite{PV1}--\cite{PV3} \cite{FM121}.
 Here we deal with {\it local minimizers}, where the topology is the $L^\infty$-topology on trajectories.

However, within the  penalization approach, these results have not been generalized to higher-order necessary conditions. Instead, through the set-separation approach, one  can establish the following result (see \cite{MPR} and \cite{MPRArXiv}) concerning the minimization of unbounded control problems and their impulsive extension:

{\bf Fact 1.} (Theorem \ref{ThIsolated}) {\it A local extended   $L^\infty$-minimizer of an unbounded-impulsive optimal control problem at which an infimum gap occurs is necessarily    abnormal  for a higher-order --i.e., including Lie brackets-- Maximum Principle}.

The importance of {\bf Fact 1} clearly  relies on the fact that  an extended process might  be abnormal for a first-order principle but normal for a higher-order one \cite[Sec.5]{MPR}, which would rule out the possibility of an infimum gap (despite abnormality at first order).

 The type of infimum gap discussed so far --where the cost of an admissible extended process is strictly lower than that of nearby admissible non-extended processes-- is the most commonly studied. However, another type of infimum gap may occur: precisely, it may happen   that a local minimizer of the original problem --here referred to as a {\it strict-sense minimizer}-- fails to be a local minimizer of the extended problem. Such a gap can undermine the stability of the optimal solution under perturbations.
 The relationship between this second type of infimum gap and abnormality is less understood.  In \cite{PV1}--\cite{PV3}, following Warga's conjectures, the authors proved  that a local strict-sense $L^\infty$-minimizer with an infimum gapwith respect to an extension by relaxation is abnormal for an {\em averaged} Maximum Principle,—where the adjoint equation  is replaced by a differential inclusion depending on the whole set of control values.  In \cite{FM221} this was extended  to an abstract setting that covered relaxation and impulsive extensions. 
A further achievement,  which includes the impulsive one, can be found 
 in \cite{FM124}, where it was proved that if a  strict-sense process has a local infimum gap with respect to the $L^1$-distance between controls (rather than the $L^\infty$-distance of the trajectories), then it satisfies the same abnormal Maximum Principle as its extended counterpart. This result broadens the abnormality-gap correspondence to a wider class of minimizers. 
 
 Still, all these findings on strict-sense minimizers  rely on first-order principles derived via  the penalization approach. Instead, thanks to the set-separation approach and the compatibility issue described in the previous subsection, in the present  paper  we can  achieve results connecting the occurrence of a strict-sense minimizers having an infimum gap with the abnormality of the {\it higher-order} Maximum Principle. 

Precisely, with reference to the impulsive extension of unbounded control systems and the associated higher-order Maximum Principle  introduced in Section \ref{Sprel} below, we establish the following result:
 

 {\bf Result 2.} (See Theorem \ref{Thstrict_sense}).  {\it If the target set is as in the hypotheses of Result 1 and we consider  the Clarke tangent cone to the target (at the final point of the optimal trajectory), then  any local strict-sense  $L^1$--minimizer with an infimum gap is   abnormal for  the higher-order Maximum Principle.}
 
 The proof of {\bf Result 2}  begins by showing that a local  $L^1$-strict-sense minimizer with an infimum gap is necessarily the limit (with respect to the $L^1$-distance between controls) of a sequence of admissible extended control-trajectory pairs, each of which  exhibiting a local infimum gap. Then, in view of {\bf Fact 1}, we associate to each of these approximating extended processes  an adjoint variable for which the higher-order Maximum Principle holds in abnormal form. The crucial step in the proof is therefore to show that the limit of these adjoint variables yields an  adjoint variable with respect to which the reference strict-sense minimizer is abnormal for the higher-order Maximum Principle.  

Let us observe that the  idea of using {\bf Result 1} and the properties of QDQ approximating cones to combine the two classical approaches based on set-separation and penalization respectively,   goes beyond the specific context of higher-order necessary conditions for infimum gap. For instance, it allows the derivation of a higher-order maximum principle based on the Clarke tangent cone starting from a higher-order maximum principle obtained via set-separation (and one may likely allow nonsmooth data, based on the results in \cite{AR25}). Moreover, state constraints could probably be incorporated by applying the two approaches at different stages of the analysis. 

The present paper is organized as follows. Section \ref{Sprel} provides key definitions and preliminary results. In Section \ref{S_T_QDQ}, we introduce the notion of a QDQ approximating cone and identify a class of sets for which the Clarke tangent cone also serves as a QDQ approximating cone. In Section \ref{Sgap}, we formulate an unbounded control minimum problem and discuss the two types of local infimum gaps introduced earlier. In particular, Subsection \ref{Sex} presents an illustrative example of these two types of gaps. Finally, Section \ref{mainsection} states {\bf Fact 1} and presents the main result concerning infimum gaps for strict-sense minimizers.


	\section{Preliminaries}\label{Sprel} 
	\subsection{Basic notation}\label{SecNotation}
  For every set $\mathcal A$ and every  subset ${\mathcal B}\subseteq\mathcal A$, we use ${\bf 1}_{\mathcal B}$ to denote the characteristic function of {$\mathcal B$}, namely, ${\bf 1}_{\mathcal B}(x)=1$  if $x\in {\mathcal B}$ and ${\bf 1}_{\mathcal B}(x)=0$ if $x\in {\mathcal A}\setminus{\mathcal B}$ .	We  use the notation  $\cR_+:=[0,+\infty)$ and $\rr_-:=(-\infty,0]$, and for any pair $a$, $b\in\cR$, we set $a\land b:=\min\{a,b\}$. $\mathbb{N}$ will denote the set of natural numbers, including $0$. 
If $N$ and $n$ are positive natural numbers,    $Lin(\rr^N,\rr^n)$ is the set of linear maps from $\rr^N$ into $\rr^n$ and $(\rr^N)^*:=Lin(\rr^N,\rr)$. Using the usual identification between $\rr^N$ and the dual $(\rr^N)^*$, the latter is also considered a Euclidean space. As  customary, `$\cdot$' is the scalar product in $\rr^N$. For every subset $\mathcal A\subseteq\rr^N$,  we will use $\bar{\mathcal A}$ and  $\partial \mathcal A$ to denote the closure and the boundary of $\mathcal A$, respectively. $\text{co} \mathcal A$ will be the {\it convex hull of $\mathcal A$}, that is
	the sets of all convex combinations of elements  of $\mathcal A$.
 Given an interval $I$,  we write $AC(I,\mathcal A)$  [resp. $C^0(I,\mathcal A)$,$L^1(I,\mathcal A)$,$L^\infty(I,\mathcal A)$] for the set of functions that  take values in $\mathcal A$  and are  absolutely continuous  [resp., continuous, Lebesgue integrable, Lebesgue measurable and essentially bounded] functions.
	We will  use 
	$\|\cdot\|_{L^\infty(I,\rr^N)}$,  and $\|\cdot\|_{L^1(I,\rr^N)}$ to denote 
	the $L^\infty$-norm and the $L^1$-norm, respectively. When no confusion may arise, we simply write  $\|\cdot\|_{L^\infty(I)}$, $\|\cdot\|_{L^1(I)}$, or also $\|\cdot\|_\infty$,  $\|\cdot\|_1$.
	$B^N_r({\check x})$  [resp. $\bar B^N_r({\check x})$] is the open  [resp. closed] ball in $\rr^N$ with center $\check x$ and  radius $r$. When $\check x=0$, we   write $B^N_r$  [resp. $\bar B^N_r$] instead of 	$B^N_r(0)$  [resp. $\bar B^N_r(0)$]. (When the dimension is clear from the context, we will  omit the superscript $N$.)

	\subsection{Cone transversality} 
	Let us recall some elementary notions concerning cones of $\rr^n$  (see e.g.  \cite{Suss1,sus}). A subset $\K\subseteq \rr^n$   is a {\it cone} if $\alpha k\in \K$  for all $(\alpha,k)\in \rr_+\times \K$. If $\mathcal A\subset \rr^n$ is any subset, let us define the  {\it conic hull} $\text{ span}^+\mathcal A$  and the {\it  polar $\mathcal A^{\bot}$ of $\mathcal A$ } 
 by setting$$
	\begin{array}{ll}
		\text{span}^+\mathcal A   &\doteq \left\{\sum_{i=1}^\ell \alpha_i v_i: \  \ell\in \mathbb{N},\ \alpha_i\geq 0, \,\,\,v_i\in \mathcal A,\,\,  \forall i=1,\ldots,\ell \right\}\subset \rr^n, \\[1.0ex] 
	\qquad	{\mathcal A}^{\bot} &\doteq\left\{p\in (\rr^n)^* :\   p\cdot w \leq 0 \ \ \forall \;
		w\in \mathcal A \right\}\subset (\rr^n)^*.
	\end{array}
	$$
	{ Observe that both}  $\text{span}^+\mathcal A$ and ${\mathcal A}^{\bot}$ are { convex cones,} while ${\mathcal A}^{\bot}$ is also closed.
\begin{definition}
 Let   $\K_1$, $\K_2\subseteq \rr^n $ be convex cones. We say that 
	$\K_1$ and $\K_2$ are {\em transversal}, if $
	\K_1-\K_2:=\big\{k_1-k_2 :\ (k_1,k_2)\in  \K_1\times \K_2\big\} = \rr^n$. 
	$\K_1$ and $\K_2$  are {\em strongly transversal}
	if they are transversal  and $\K_{1}\cap\K_{2} \supsetneq\{0\}$.\end{definition}
	
	\begin{prop}\label{teo2} Two  convex cones $\K_1$, $\K_2\subseteq \rr^n $ are transversal if and only if 
		either  they are strongly transversal  or they
		are complementary linear subspaces, namely  $\K_1\oplus{\K}_2=\rr^n$  (i.e., $\K_1 +\K_2=\rr^n$ and $\K_1 \cap\K_2=\{0\}$).
	\end{prop}
	
	\begin{prop}\label{teo3}Two  convex cones $\K_1,\K_2\subseteq \rr^n$ are not transversal  if and only if $\K_1$ and $\K_2$ are  {\rm linearly separable}, by which we mean that   $(-\K_1^\bot\cap\K_2^\bot) \backslash \{0\} \neq \emptyset,$ namely, there exists a  linear form $\lambda\in (\rr^n)^*\backslash \{0\}$
		such that 
		$ \lambda \cdot k_1\geq 0$  
		and $\lambda \cdot k_2\leq 0$, $\forall (k_1,k_2)\in \K_1\times  \K_2$,.
	\end{prop} 
	
	\subsection{Some generalized normal and tangent cones} \label{subGenCone}

	In this subsection we collect some  basic concepts  from nonsmooth analysis (see e.g. \cite[Chapters 1, 2]{ClarkeLed}).
	\vsmm
 We recall that, if   $\A\subset \rr^n$,   
 for any  $x\in \rr^n$  we  use $d_{\A}(x) $ to denote the {\it distance of $x$ from $\A$}, namely we set $d_{\A}(x):=\inf\{|x-y|:  \ y\in \A\}$.

  Let $\mathcal{S}\subset \rr^n$ be a closed, nonempty set and let $x\in  \rr^n$. We use  $\Pi_\mathcal{S}(x)$ to denote the {\em metric projection of $ x$ on $\mathcal{S}$}, i.e. the  set $$\Pi_\mathcal{S}( x):=\{y\in\mathcal{S}: \ | x-y|=d_{\mathcal{S}}(x)\}.$$ Notice that, since  $\mathcal{S}$ is closed (and nonempty), $\Pi_\mathcal{S}( x)\ne\emptyset$ for all $ x\in \rr^n$.

Let   $\bar x\in \mathcal{S}$.   The \textit{proximal normal cone $N^P_\mathcal{S}(\bar x)$ to $\mathcal{S}$ at $\bar x$} is defined as
\[
N^P_\mathcal{S}(\bar x)  := \left\{  \eta\in \rr^n \text{ : } \exists r>0 \,\,\text{ such that }\,\, \bar x\in  \Pi_\mathcal{S}(\bar x+r\eta)   \right\}. \footnotemark
\]
  \footnotetext{ Here and in the following, we are  using the convention of identifying either a tangent space  $T_{\bar x}\rr^n$   and a  cotangent  space $T^*_{\bar x}(\rr^n)$  with  $\rr^n$. In particular, $N^P_\mathcal{S}(\bar x)$, which is here defined as a subset of $\rr^n$, should be thought as a subset of $T^*_{\bar x}(\rr^n)$. }

One  can check that $\eta\in N^P_\mathcal{S}(\bar x)$ if and only if there exists some $r>0$ such that
\bel{N^Pi}
\eta \cdot (x-\bar x) \leq \frac{1}{2r} |x-\bar x|^2 \quad  \forall x \in \mathcal{S}.
\eeq
Furthermore, for any $y\in \rr^n$, one has
$
y-\bar x\in N^P_\mathcal{S}(\bar x)$ for all $\bar x\in \Pi_\mathcal{S}(y).
$
 Let us point out that if $S$ is a convex subset then, for every $\bar x\in S$,  $$N^P_\mathcal{S}(\bar x) = N^{conv}_\mathcal{S}(\bar x)\colon= \Big\{\eta\in  \rr^n\colon\,\,\,\, \eta\cdot(x-\bar x)\leq 0 \,\,\forall x\in \mathcal{S} \Big\},$$ i.e. the proximal normal cone $N^P_\mathcal{S}(\bar x)$ coincides with the {\it normal cone $ N^{conv}_\mathcal{S}(\bar x)$} of convex analysis. 
Finally, it is worth pointing out that the  concept of proximal normal cone is local, namely $N^P_\mathcal{S}(\bar x)=N^P_{\mathcal{S}\cap V}(\bar x)$  for any neighborhood $V$ in $\rr^n$ of $\bar x\in \mathcal{S}$. 

	The \textit{limiting normal cone} $N_\mathcal{S}(\bar x)$ to $\mathcal{S}$ at $\bar x$ is given by 
\[ 
N_\mathcal{S}(\bar x) := \left\{  \eta\in \rr^n \text{ : } \exists x_i \stackrel{\mathcal{S}}{\to} \bar x,\, \eta_i \to \eta \,\,\text{ s.t. }\,\,  \eta_i \in N^P_\mathcal{S}(x_i) \text{ for each } i\in \mathbb{N}     \right\},
\]
 in which the notation $x_{i} \stackrel{\mathcal{S}}{\longrightarrow}\bar{x}$ is used to indicate that the sequence $(x_i)_{i\in{\mathbb N}}$  converges to $\bar x$ and   $x_i\in\mathcal{S}$ $\forall i\in{\mathbb N}$.  The proximal normal cone $N^P_\mathcal{S}(\bar x)$ is convex, and   $\{0\}\subseteq N^P_\mathcal{S}(\bar x) \subseteq N_\mathcal{S}(\bar x)$,  while the limiting normal cone $N_\mathcal{S}(\bar x)$ is   not necessarily convex (see fig.1). On the other hand, the limiting cone    $N_\mathcal{S}(\bar x)$  is  closed, while  the proximal normal cone $N^P_\mathcal{S}(\bar x)$   may happen to be not closed. Furthermore, the set-valued map $x\rightsquigarrow N_\mathcal{S}(x)$ from $\mathcal{S}$ to the set of subsets of  $\left({\rr^n}\right)^*$  is upper semicontinuous \cite{vinter}.

Let us now recall the definitions of a Clarke tangent cone and a Clarke normal cone. The {\em Clarke tangent cone to  $\mathcal{S}$ at $\bar x$} \cite{ClarkeLed} is the set 
\bel{T^c}
T_\mathcal{S}^C(\bar x):=\left\{v\in  \rr^n: \\ \underset{t\downarrow 0, \, x\to \bar x}{\limsup} \, \frac{d_\mathcal{S}(x+t v)-d_\mathcal{S}(x)}{t}=0\right\}.
\eeq
The {\em Clarke normal cone to  $\mathcal{S}$ at $\bar x$}, denoted by $N_\mathcal{S}^C(\bar x)$,  is defined as the polar  of the Clarke tangent cone, namely $$N_\mathcal{S}^C(\bar x):=(T_\mathcal{S}^C (\bar x))^{\bot}.
$$
An equivalent ``sequential" definition of  $T_\mathcal{S}^C(\bar x)$ is as follows: an element $v\in  \rr^n$ belongs to  $T_\mathcal{S}^C(\bar x)$ if and only if, for every sequence $x_i   \stackrel{\mathcal{S}}{\to} \bar x$ and every sequence $t_i\downarrow 0$, there exists a sequence $(v_i)_i\subset   \rr^n$ converging to $v$ such that $x_i+t_iv_i\in \mathcal{S}$ for all $i$.

For the sequel, it is useful  to prove also the following, third, equivalent definition. 
\begin{lma}\label{Lem_T^c} Given  a closed, nonempty subset  $\mathcal{S}\subset \rr^n$  and a point  $\bar x\in \mathcal{S}$, we have
\bel{T^c_new}
T_\mathcal{S}^C(\bar x)=\left\{v\in  \rr^n: \ \ \lim_{t\downarrow 0, \, x\stackrel{\mathcal{S}}{\to} \bar x}\frac{d_\mathcal{S}(x+t v)}{t}=0\right\}.
\eeq
\end{lma}
\begin{proof}  Preliminarily observe that for any $v\in\rr^n$ and for every $x\in \mathcal{S}$, 
since $d_\mathcal{S}(x)=0$, we have
\begin{equation}\label{L2.1_1}
0\le \limsup _{t\downarrow 0, \, x\stackrel{\mathcal{S}}{\to} \bar x}\frac{d_\mathcal{S}(x+t v)}{t}= \limsup _{t\downarrow 0, \, x\stackrel{\mathcal{S}}{\to} \bar x}\frac{d_\mathcal{S}(x+t v)-d_\mathcal{S}(x)}{t}
\le \limsup _{t\downarrow 0, \, x \to \bar x}\frac{d_\mathcal{S}(x+t v)-d_\mathcal{S}(x)}{t}.
\end{equation}

Take $v\in T_\mathcal{S}^C(\bar x)$. Since $\ds\limsup_{t\downarrow 0, \, x\to \bar x}\frac{d_\mathcal{S}(x+t v)-d_\mathcal{S}(x)}{t}=0$ by definition,  from \eqref{L2.1_1} it follows immediately that    
$$
0\le  \liminf _{t\downarrow 0, \, x\stackrel{\mathcal{S}}{\to} \bar x}\frac{d_\mathcal{S}(x+t v)}{t}\le \limsup _{t\downarrow 0, \, x\stackrel{\mathcal{S}}{\to} \bar x}\frac{d_\mathcal{S}(x+t v)}{t}=0.
$$
 Thus, there exists  $\ds\lim_{t\downarrow 0, \, x\stackrel{\mathcal{S}}{\to} \bar x}\frac{d_\mathcal{S}(x+t v)}{t}=0$. 

In contrast, fix $v\in  \rr^n$ satisfying
\bel{contrad}
\lim_{t\downarrow 0, \, x\stackrel{\mathcal{S}}{\to} \bar x}\frac{d_\mathcal{S}(x+t v)}{t}=0
\eeq
 and suppose, by contradiction, that $v\notin T_\mathcal{S}^C(\bar x)$. Then, there exist some sequences $(t_i)_i\subset(0,+\infty)$ and $(x_i)_i\subset  \rr^n$ such that $t_i\downarrow 0$, $x_i\to \bar x$, and 
$$
\lim_i\,\frac{d_\mathcal{S}(x_i+t_i v)-d_\mathcal{S}(x_i)}{t_i}=: \ell\ne 0,
$$
 where $\ell>0$ by \eqref{L2.1_1}. Then, 
for each $i\ge 1$, choose $y_i\in \mathcal{S}$ that satisfies
\bel{y_iS}
|y_i-x_i|\le d_\mathcal{S}(x_i)+\frac{t_i}{i}.
\eeq
Hence, we obtain a sequence $(y_i)_i\subset \mathcal{S}$ converging to $\bar x$ and  such that
$$
\begin{array}{l}
\ds\frac{d_\mathcal{S}(y_i+t_i v)}{t_i}=\frac{d_\mathcal{S}(y_i+t_i v) -d_\mathcal{S}(x_i+t_i v)+d_\mathcal{S}(x_i+t_i v)}{t_i} \\
\ds\qquad\qquad\qquad\ge \frac{-|y_i-x_i|+d_\mathcal{S}(x_i+t_i v)}{t_i} 
 \ge \frac{d_\mathcal{S}(x_i+t_i v)-d_\mathcal{S}(x_i)}{t_i}-\frac{1}{i},
\end{array}
$$
owing to \eqref{y_iS} and  the $1$-Lipschitz continuity of the distance function $d_\mathcal{S}$. Passing to the limit, we get $\lim_i\frac{d_\mathcal{S}(y_i+t_i v)}{t_i}\ge\ell$, in contradiction to \eqref{contrad}.  
\end{proof} 

Let us conclude recalling some useful  relations   between the above cones. The Clarke tangent cone $T_\mathcal{S}^C(\bar x)$  to  $\mathcal{S}$ at $\bar x\in \mathcal{S}$  is a closed convex cone, which in turn is  the polar of the  Clarke normal cone $N_\mathcal{S}^C(\bar x)$. The latter and the limiting normal cone are related by the equality: \bel{polare_T^c}
(T_\mathcal{S}^C(\bar x))^{\bot}=N_\mathcal{S}^C(\bar x)=\overline{\text{co}\, N_\mathcal{S}(\bar x)}.
\eeq
 
\begin{figure}[h!]\centering\includegraphics[scale=0.3]{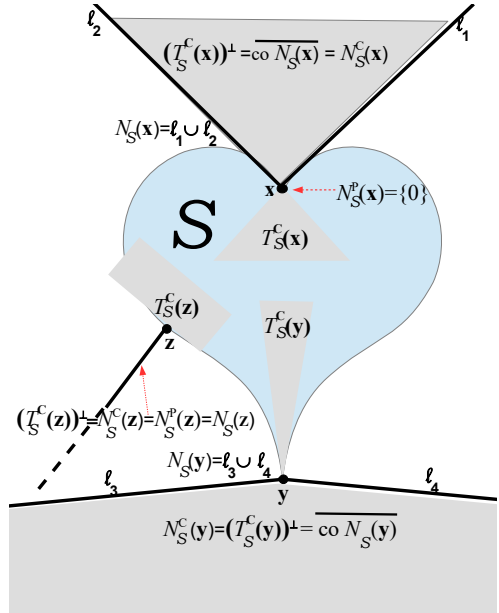}\caption{Tangent and normal cones}\end{figure}

\subsection{Prox-regular sets}\label{pr_sect}
This subsection is devoted to some basic facts on { the notion of} {\it prox-regularity},  which  will be used  in the next sections. For  more details
	on this topic, we refer, for instance, to \cite{PRT00}, the survey \cite{CT10},  and the references therein. It was Federer \cite{Fed59} who first introduced prox-regular sets in finite-dimensional spaces as {\em sets with positive reach}, in order to provide a generalization of both convexity and $C^2$-smoothness.
 
\begin{definition}\label{Def_prox-reg} Let $\mathcal{S}\subset\rr^n$ be a closed nonempty subset  and consider a real number  $r>0$.  One says that $\mathcal{S}$ is {\em $r$-prox-regular} if, for each $\bar x\in \mathcal{S}$,  for all $\eta\in N^P_\mathcal{S}(\bar x)$ such that $|\eta|=1$,  and for every real number $s\in(0,r]$,   one has $\bar x\in \Pi_\mathcal{S}(\bar x+s\eta)$.
\end{definition}

 \begin{remark} Given a  closed nonempty subset $\mathcal{S}\subset\rr^n$, $\bar x\in \mathcal{S}$ and $\eta\in N^P_\mathcal{S}(\bar x)$ such that $|\eta|=1$, it is easy to show  that for any $s>0$ one has
\bel{N^Piii}
		\bar x\in \Pi_\mathcal{S}(\bar x+s\eta)\ \ \Longleftrightarrow  \ \mathcal{S}\cap B_{s}(\bar x+s\eta)=\emptyset.
\eeq
\end{remark}

 The following proposition states some  characterizations of $r$-prox-regular sets,  for which we refer to \cite{PRT00,CT10}.

 \begin{prop}\label{Prop_prox_reg} Given a nonempty closed subset  $\mathcal{S}$ of $ \rr^n$ and $r>0$, the following properties are equivalent.
 
 \begin{itemize}
 \item[{\rm (i)}] $\mathcal{S}$ is $r$-prox-regular;
 \item[{\rm (ii)}]  for all $x$, $x'\in \mathcal{S}$, for every $\eta\in N^P_\mathcal{S}(x)$, one has
 $$
 \eta\cdot(x'-x)\le \frac{1}{2r}|\eta|\,|x'-x|^2;
 $$

\item[{\rm (iii)}]  on the set $\mathcal{U}_r:=\left\{x\in  \rr^n: \ d_\mathcal{S}(x)<r\right\}$,  the projection  $\Pi_\mathcal{S}$ is univalued.  Moreover, one has that for every $s\in(0,r)$,
$$
\big|\Pi_\mathcal{S}(x)-\Pi_\mathcal{S}(y)\big|\le \frac{1}{1-(s/r)}|x-y|,  \quad \text{for all} \ x,y\in \mathcal{U}_s
$$
(thus, $\Pi_\mathcal{S}$ is a  Lipschitz continuous function on any $\mathcal{U}_s$, $s<r$); 
 \item[{\rm (iv)}]  for every $x\in \mathcal{U}_r\setminus  \mathcal{S}$, one has that ($\Pi_S$ is univalued and)
 $$
 \Pi_\mathcal{S}(x)=\Pi_\mathcal{S}\left(\Pi_\mathcal{S}(x)+s\frac{x-\Pi_\mathcal{S}(x)}{|x-\Pi_\mathcal{S}(x)|}\right), \quad\text{for any $s\in[0,r)$;}
 $$
  \item[{\rm (v)}] on the set $\mathcal{U}_r$ the function $d^2_\mathcal{S}$ is differentiable  with locally Lipschitz derivative, and
  $$
  \nabla   d^2_\mathcal{S}(x)=2(x-\Pi_\mathcal{S}(x)),  \quad \text{for all} \ x\in \mathcal{U}_r.
  $$
 \end{itemize}
  \end{prop}

	\section{QDQ approximating cones and  Quasi   Prox Regularity}\label{S_T_QDQ}
	\subsection{Quasi Differential Quotients {\rm (QDQ)} and conic approximation} In this subsection we recall the definition --introduced in \cite{PR} (see also \cite{APR25, AR25})-- of {\it Quasi Differential Quotient  } (QDQ) of a set-valued map  and the corresponding notion of {\it QDQ approximating cone} to a subset $\mathcal E\subset\rr^n$ 
	   at a  point of $\mathcal E$.

	As customary, we  call a  function   $\rho : \cR_+\to\cR_+$   a   {\it modulus}   if it  is monotonically nondecreasing and
 	$\ds\lim_{s\to 0^+}\rho(s) =\rho(0)= 0$.

	\begin{definition}
	\label{qdq}
		Let $G : \cR^N \rightsquigarrow \cR^n$     be a 
		set-valued map (for some positive integers $n,N$), and consider  $(\bar  \xi,\bar y) \in \cR^N\times\cR^n  $. Let $\Lambda\subset Lin(\cR^N, \cR^n)$  be  a compact set and let $\Gamma\subset\cR^N$  be  any  subset.
		We
		say that $\Lambda$ is a {\rm Quasi Differential Quotient  (QDQ) of $G$ at  $(\bar  \xi,\bar y)$  in the direction of $\Gamma$}  if   there
		exist  a modulus $\rho$ and some $\bar\delta>0$ enjoying  the property that,
		for any $\delta\in[0,\bar\delta]$,  there is a continuous map  
		$(L_\delta,h_\delta):   B_{\delta}(\bar{ \xi})\cap\Gamma  \to Lin(\cR^N, \cR^n) \times \cR^n$ 
		such that, for all $ \xi
		\in   B_\delta(\bar{ \xi})\cap \Gamma$,\footnote{Notice that we are not assuming that $\bar \xi\in \Gamma$, so that it may well happen that  $\bar y\notin G(\bar\xi)$.}
		$$
		\bar y +  L_\delta ( \xi)\cdot( \xi-\bar  \xi)  + h_\delta( \xi)\in G( \xi), \quad \min_{L'\in\Lambda}|L_\delta( \xi) - L'|\leq \rho(\delta), 
		\quad |h_\delta( \xi)|\leq \delta \rho(\delta).  
		$$
	\end{definition}
\begin{remark}\label{Rem_QDQ}
Using the notations of Definition \ref{qdq}, one can observe that,  if for some $\bar\delta>0$ there is a continuous map $(L,h):  B_{\bar\delta}(\bar{ \xi})\cap\Gamma  \to Lin(\cR^N, \cR^n) \times \cR^n$  
		such that
		\bel{C_lin}
		\bar y +  L  ( \xi)\cdot( \xi-\bar  \xi)  + h ( \xi)\in G( \xi),   \quad \text{for all }  \xi
		\in   B_{\bar\delta}(\bar{ \xi})\cap \Gamma
		\eeq
and
\bel{C_sup}
\sup_{ \xi\in B_{\delta}(\bar{ \xi})\cap \Gamma}\min_{L'\in\Lambda}|L ( \xi) - L'|\to 0, 
		\quad \frac{\sup_{ \xi\in B_{\delta}(\bar{ \xi})\cap \Gamma} |h ( \xi)|}{\delta}\to0 \quad \text{as} \ \delta\to0^+, 
		\eeq
		then  $\Lambda$ is a QDQ of $G$  at  $(\bar  \xi,\bar y)$  in the direction of $\Gamma$ (see \cite[Remark 2.2]{APR25}).
		\end{remark}
		


	

\begin{definition}
	\label{ApprCone}    Consider a subset  $\mathcal S\subset\rr^n$ and $\bar x\in \mathcal S$. A
	convex cone $\K\subset\rr^n$ is called a   {\rm  QDQ approximating cone  to $\mathcal S$ at $\bar x$} if there exist an integer   $N\ge0$, a set-valued
		map  $G : \cR^N  \rightsquigarrow \rr^n$, a convex cone $\Gamma\subset\cR^N$, and a  Quasi Differential Quotient  $\Lambda:=\{L\}$  for some $L\in \text{Lin}(\cR^N,\cR^n)$ of $G$  at  $(0,\bar x)$ in the direction of $\Gamma$ such that $G(\Gamma)\subset \mathcal S$ and $\K =
		L\cdot\Gamma$. 
		We say that such a {\rm triple $(G,\Gamma,\Lambda)$  generates the  cone $\K$}.

	\end{definition}

{}
	
	
For instance, the  classical  Boltyanski approximating cone  is a  special case of QDQ approximating cone (as well as other well-known approximating cones, see e.g. \cite{APR25})
 
	\subsection{Quasi Prox-Regularity and Clarke tangent cones}

{\begin{definition}\label{defprox}  Let  $\mathcal{S}\subset\cR^n$ be a nonempty, closed subset. Fix $\bar x\in \mathcal{S}$.  We say that $\mathcal{S}$ is {\rm Quasi Prox-Regular at $\bar x$} if there exists a real number $\delta>0$   such that  either
	\begin{itemize} \item[{\rm (i)}] $(\bar x+T^C_\mathcal{S}(\bar x))\cap B_\delta(\bar x)\subseteq \mathcal{S}$ or \item[{\rm (ii)}]  for some $r>0$ there exists 
 an $r$-prox regular set $\tilde{\mathcal{S}}$ such that $\mathcal{S}\cap  B_\delta(\bar x) = \tilde{\mathcal{S}}\cap  B_\delta(\bar x)$.\end{itemize}
 \end{definition}


%
\begin{thm}\label{Th_T^C} For some  $n\in{\mathbb N}\backslash\{0\}$,  let $\mathcal{S}\subset\cR^n$ be a nonempty, closed subset, and  fix  $\bar x\in \mathcal{S}$. If $\mathcal{S}$ is { Quasi Prox-Regular at $\bar x$}, then the Clarke tangent cone $T^C_\mathcal{S}(\bar x)$ to $\mathcal{S}$ at $\bar x$ is a QDQ approximating cone to $\mathcal{S}$ at $\bar x$.
\end{thm} 
\begin{proof}  
Up to a translation, it is not restrictive to assume that $\bar{x}=0$.  Define the set-valued map $G : \cR^n  \rightsquigarrow\cR^n$, given by
$$
G(x):=
 \Pi_\mathcal{S}(x), \quad \text{for all} \ x\in  \cR^n.
$$
Consider the closed, convex cone $\Gamma:=T^C_\mathcal{S}(0)\subset\cR^n$ and set $\Lambda=\{L\}$ with $L:=\mathbb{1}_n$, namely the identity matrix on $\cR^n$.   We claim that proving that $\K:=T^C_\mathcal{S}(0)$ is  a QDQ approximating cone to $\mathcal{S}$ at $0$ reduces to verifying that 
{\it \begin{itemize}
\item[{\rm (A)}] there exist some $\bar \delta>0$  and a continuous  selection $B_{\bar\delta}\cap T^C_\mathcal{S}(0): {x\to} \alpha(x)\in G(x)$, whereby the (continuous) function $h:B_{\bar\delta}\cap T^C_\mathcal{S}(0)\to \cR^n$, defined as 
$$
h(x):=\alpha(x)-x,
$$
for each $\delta\in(0,\bar\delta]$, satisfies
\bel{A_h}
\frac{\ds\sup_{x\in B_{\delta}\cap T^C_\mathcal{S}(0)}|h(x)|}{\delta}\to 0 \quad \text{as $\delta\to0^+$} .
\eeq
\end{itemize}}
Indeed,  assuming (A) to be valid, we have that, for all $x\in B_{\bar\delta}\cap T^C_\mathcal{S}(0)$,
\bel{Lin_T^C}
0+L\cdot x +h(x)=\mathbb{1}_n\cdot x+\alpha(x)-x=\alpha(x)\in G(x).
\eeq
Therefore, owing to Remark \ref{Rem_QDQ}, $\Lambda=\{L\}$ is a QDQ  of $G$ at $(0,0)$ in the direction of  $\Gamma$, since \eqref{Lin_T^C}, \eqref{A_h} easily imply \eqref{C_lin} and \eqref{C_sup},  respectively, as soon as $h$ is as in (A) and $L\equiv \mathbb{1}_n$.  To conclude that the convex cone $\K=T^C_\mathcal{S}(0)$ is a QDQ approximating cone to $\mathcal{S}$ at $0$, it remains only to show that  
$G(\Gamma)\subset \mathcal S$ and $\K =
		L\cdot\Gamma$.
 But these conditions are trivially satisfied  by definition, as, in fact,  $G(x)=\Pi_{\mathcal{S}}(x)\subseteq {\mathcal{S}}$ for all $x \in\cR^n$, while   $\K =
		L\cdot\Gamma$  now reads $T^C_\mathcal{S}(0) =\mathbb{1}_n\cdot T^C_\mathcal{S}(0)$. 
		
The proof is thus concluded in the case in which we show that property (A) holds true at $0$ in $\mathcal{S}$  if either condition {\rm (i)} or condition {\rm (ii)} of Definition \ref{defprox} is satisfied.

In case condition  {\rm (i)} is met, namely, if $B_{\bar\delta}\cap T^C_\mathcal{S}(0)\subset \mathcal{S}$   for some $\bar\delta>0$, property (A) is trivially verified. Indeed $\Pi_\mathcal{S}(x)=\{x\}$ for all $x\in B_{\bar\delta}\cap T^C_\mathcal{S}(0)$, i.e. there exists an unique continuous selection  $\pi(x)=x$,  for every $x\in B_{\bar\delta}\cap T^C_\mathcal{S}(0)$. So $h(x)=\pi(x)-x= 0$ for all $x\in B_{\bar\delta}\cap T^C_\mathcal{S}(0)$.

Assume now {\rm (ii)}, namely   there exist some $\hat\delta,r>0$,  and a $r$-prox-regular set $\tilde{\mathcal{S}}\subset\cR^n$ such that $\mathcal{S}\cap  B_{\hat\delta}=\tilde{\mathcal{S}}\cap  B_{\hat\delta}$. By Proposition \ref{Prop_prox_reg}, for any  $\bar r\in(0,r)$   the projection  $\Pi_{\tilde{\mathcal{S}}}$ is univalued  and Lipschitz continuous on the set $\mathcal{U}_{\bar r}=\left\{x\in  \tilde{\mathcal{S}}: \ d_{\tilde{\mathcal{S}}}(x)<\bar r\right\}$.  Take $\bar r:=\min\{r/2,\hat\delta/2\}$ and set $\bar\delta:=\bar r/2$. As it is easy to check, we have that   $T^C_\mathcal{S}(0)=T^C_{\tilde{\mathcal{S}}}(0)$ and $\Pi_{\mathcal{S}}(x)=\Pi_{\tilde{\mathcal{S}}}(x)$ for all $x\in  B_{\bar\delta}$.  As a consequence, { on $ B_{\bar\delta}\cap T^C_\mathcal{S}(0)$ ($\subset \mathcal{U}_{\bar r}$), there exists a trivial  unique, Lipschitz continuous,  selection  $\alpha(x)\in \{\alpha(x)\} = \Pi_{\mathcal{S}}(x)$ },   so   the function $h(x)=\alpha(x)-x$ is well-defined and continuous on this set. To prove that $h$ satisfies condition  \eqref{A_h}, for any $\delta\in(0,\bar\delta]$, let $x_\delta\in  \overline{B_{\delta}}\cap T^C_\mathcal{S}(0)$ be the point such that 
$$
|h(x_\delta)|=\sup_{x\in B_{\delta}\cap T^C_\mathcal{S}(0)}|h(x)|,
$$
which exists  since $h$ is continuous and   $T^C_\mathcal{S}(0)\cap\overline{B_\delta} $  is compact.  Suppose, by contradiction, that there exists a sequence $\delta_i\downarrow 0$ as $i\to+\infty$, such that 
\bel{contraAh}
\lim_i \frac{|h(x_i)|}{\delta_i}=\lim_i \frac{|\pi(x_i)-x_i|}{\delta_i}=\ell>0,
\eeq
where  $x_i:=x_{\delta_i}$.  Since $x_i\in B_{\delta_i}\cap T^C_\mathcal{S}(0)$,  we  have that   for each $i$,  
$$
x_i=t_i\,v_i,
$$
 for some $t_i\in(0,\delta_i]$ \footnote{  Possibly passing  to a subsequence, we can assume   $t_i>0$.}  and  $v_i\in T^C_\mathcal{S}(0)$ with $|v_i|=1$.  Possibly passing   to a subsequence (we do not relabel), there exists some $v\in  T^C_\mathcal{S}(0)$ with $|v|=1$ such that $v_i\to v$.   We get
 $$
 \begin{array}{l}
|h(x_i)|=|\alpha(x_i)-x_i| =|\alpha(t_i\,v_i)-t_i\,v_i|\le |\alpha(t_i\,v)-t_i\,v| \\
\ \quad\qquad\qquad\qquad\qquad +|\alpha(t_i\,v_i)-t_i\,v_i-\alpha(t_i\,v)+t_i\,v| \\
\ \quad\qquad\qquad\qquad\qquad\le d_{\mathcal{S}}(t_i\,v)+3t_i|v-v_i|,
\end{array}
 $$
 where we have used  the $2$-Lipschitz continuity of $\alpha$ (which follows from Proposition \ref{Prop_prox_reg}  (iii), being $\bar r\le r/2$). Then, from the characterization  \eqref{T^c_new} of the Clarke tangent cone (and the fact that $t_i\le\delta_i$) it follows that 
 $$
 \frac{|h(x_i)|}{\delta_i}\le \frac{d_{\mathcal{S}}(t_i\,v)}{t_i}+3|v-v_i| \to 0 \quad \text{as $i\to+\infty$,}
 $$
 in contradiction to \eqref{contraAh}. 
\end{proof}

	\section{Impulsive extension of a minimum problem  and infimum gap phenomena}\label{Sgap}  
	
	 In this Section we  begin  the application of   the material discussed in the previous sections  to the study of gap phenomena in impulsive optimal control. 
		\subsection{An unbounded control problem and its impulsive extension}
We  will consider an optimal control problem on $\rr^n$ ($n\geq 1$) with a  (space-time)  endpoint constraint, a  closed subset  $\cS\subseteq[0,+\infty)\times \rr^n$, which we  call the   {\it target}. The  {\it cost} is a function $\h:[0,+\infty)\times \rr^n\to \rr$, while the  set $\U$ of (unbounded) {\it strict-sense controls} is defined as
  	$$\U:=\bigcup_{T>0} \left(\{T\}\times L^1\big([0,T] , \C \big)\right),$$
  where $\C\subseteq\cR^m $  is a closed cone. 
For any  strict-sense control $(T,u)\in\U,$ we call $(T,u,x,\vz)$ a {\it strict-sense process} if  $(x,\vz)$     is  the (unique)  Carath\'eodory solution on $[0,T]$ to the Cauchy problem 
\bel{E2}
\left\{\begin{array}{l}
\ds \frac{dx}{dt}(t) = f(x(t)) + \ds\sum_{i=1}^m  g_i(x(t)) {u^{i}}(t)\\
\ds\frac{d\vz}{dt}(t) = |u(t)|\qquad
\\ [1.5ex] 
(x,\vz)(0)=({\check x},0),\end{array}\right.
\eeq 
  where  $f$,  $g_1,\ldots,g_m$ are given vector fields.  At every $t\in [0,T]$, $\nu(t)$ coincides with the $L^1$-norm of the control $u$ on $[0,t]$, which   can be regarded as the  {\it energy } spent up to time $t$. We possibly  impose  an {\it energy  upper bound} $\nu(T)\leq K$, with $K\in [0,+\infty[\cup {+\infty}$  (and  $K=+\infty$ will  mean the absence of such a bound).  A strict-sense process $(T,u,x,\vz)$ is said {\em  feasible} if  $(T, x(T))\in\cS$ and $\vz(T)\le K$. 
  We will call 
 \begin{equation}\tag{P}\label{P}
		\ds\inf  \Big\{\h(T,x(T)): \quad (T,u,x,\vz) \ \text{ feasible strict-sense process}\Big\}.
\end{equation} the {\it original 
optimal control problem}.
Since the controls $u$ are unbounded and no  coerciveness  conditions avoid the occurrence of minimizing sequences of strict-sense trajectories   which converge  to  discontinuous paths (which  would impede the existence of a minimizer among strict-sense processes), following the {\it graph completion} approach
 (see e.g. \cite{BR,MR,MiRu,BP,GS,KDPS,AKP1,WZ}) we   embed problem \eqref{P}  into the  {\it extended  optimal control problem}
\begin{equation}
 \tag{\text{$\mathcal{P}_{\scriptsize e}$}}
\label{Pe}
	\inf \Big\{\h(y^0(S), y(S)): \ \ (S,w^0,w,y^0,y,\beta) \ \text{ feasible extended  process} \Big\},
\end{equation}	
where $(S,w^0,w, y^0,y,\beta)$ is  called {\it extended   process} when $(S,w^0,w)$ belongs to  the  set $\mathcal{W}$  of {\it extended   controls},  defined as 
$$
\mathcal{W}:=
\bigcup_{S>0}\left( \{S\}\times \Big\{(w^0,w)\in L^\infty ([0,S], \cR_+\times \C): \  w^0+|w|=1 \ \text{a.e.} 
  \Big\}\right),
$$
and  $(y^0,y,\beta)$ is  the unique  solution on $[0,S]$ to 
 \begin{equation}
 \label{extended}
 \left\{
\begin{split}
\frac{d{y^0}}{ds} (s) &= w^0(s),  \\
\frac{dy}{ds} (s) & = f( y(s) )w^0(s)+ \sum_{i=1}^{m}g_{i}( y(s))w^{i}(s),\\
\frac{d\beta}{ds} (s) & = |w(s)|, \\
 (y^0,y , &\beta)(0)=(0,\xbo,0).
\end{split}
\right.  \end{equation}
As in the original problem, an extended process,  $(S,w^0,w, y^0,y,\beta)$ is said {\it feasible} provided   $( y^0,y,\beta)(S)\in \T\times [0,K]$.
Notice that, by regarding  $y^0$ das a reparameterization of time $t$, the class of strict-sense processes can be   identified with the subclass of extended   processes  with $w^0>0$ almost everywhere.  The elements of this subclass are  referred to as   {\em embedded strict-sense processes}.  
 Precisely, given a strict-sense process  $(T,u,x,\vz)$, through the time change  $y^0:=\sigma^{-1}$,  where $[0,T]\ni t\mapsto\sigma(t):=t+\vz(t)$, we   define the associated embedded strict-sense process $(S,w^0,w, y^0,y,\beta)$   as follows:
 \bel{seq}
(S,w^0,w, y^0,y,\beta):=\left(\sigma(T),  \frac{dy^{0}}{ds}, (u\circ y^{0})\cdot \frac{dy^{0}}{ds},  y^{0},x\circ y^{0}, \va\circ y^{0}\right).
\eeq
Conversely, given  any embedded strict-sense process  $(S,w^0,w ,y^0,y,\beta)$ we get the corresponding  strict-sense process $(T,u,x,\vz)$ using the time change {\it  $t\mapsto\sigma(t):=(y^0)^{-1}(t)$}.
This is a one-to-one correspondence from the set of strict-sense processes to the subset of embedded strict-sense processes. Within this bijection, {\it  a process  $(T,u,  x,\va)$ is feasible if and only if the associated process $(S,w^0,w, y^0,y,\beta)$ is feasible, and $\Psi(T,x(T))=\Psi(y^0(S),y(S))$}.   Therefore, the original minimum  problem  is equivalent to the following {\it embedded strict-sense  optimal control problem}
\begin{equation}
\tag{\text{$\mathcal{P}$}}\label{Ps}
	\inf \Big\{\h(y^0(S), y(S)): \ \ (S,w^0,w,y^0,y,\beta) \ \text{ feasible embedded strict-sense  process} \Big\}.
\end{equation}	
 The  above {\it impulsive extension}  ($\mathcal{P}_{\scriptsize e}$) of problem ($\mathcal{P}$) (i.e. of problem (P)) }consists in  allowing  the  time  derivative  $w^0$   to vanish on a set of positive measure. In particular, if   $I\subseteq [0,S]$ is an  interval on which $w^0$ vanishes,  at time  $t:=y^0(I)$  the state $y$ {\it evolves instantaneously}. \footnote{ Alternatively, one can provide an equivalent   $t$-based description of this extension using bounded variation trajectories, as it is done in   \cite{KDPS,AKP1,FM224}.}

	Throughout the paper,  we   assume the following hypotheses.  
	\begin{itemize}
		\item[{\rm (i)}]	{\it The vector fields  $f, g_1,\dots,g_m:\cR^n\to\cR^n$  are of class $C^1$ and bounded with their derivatives.}
		\item[{\rm (ii)}] {\it  The {\rm  final cost} $\h:\rr\times\rr^n\to\rr$ is of class  $C^1$. }
		\item[{\rm (iii)}]  {\it The control set $\C\subseteq\cR^m$ is a closed cone of the form
	$
	\mathcal{C}   = \mathcal{C} _1\times  \mathcal{C} _2 \subseteq \cR^{m_1}\times \cR^{m_2},
	$
	where  $(m_1,m_2)\in\cN^2$, $m_1+m_2=m$,    and, if $m_1\geq 1$,   $\mathcal{C} _1$ is a closed cone that  contains the lines $\{r{\bf e}_i:\ r\in\cR\}$, for $i=1,\ldots,m_1$,    while, if $m_2\geq 1$,       $\mathcal{C}_2$ is a  closed cone which does not contain any   straight line.\footnote{If $m_1=0$ [resp. $m_2=0$] we mean that 	$\mathcal{C}   = \mathcal{C} _2	$	[resp. 	$\mathcal{C}   = \mathcal{C} _1	$].}}
\end{itemize}

\begin{remark}  {\rm  The main  results of this paper are of local nature, so, given  a  feasible  extended  process $\bar z:=(\bar S,\bar w^0,\bar w , \bar y^0,\bar  y,\bar\beta)$,   they are still valid if the regularity hypothesis (i) is  replaced by the assumption that $f, g_1,\dots,g_m$ of class $C^1$ in an $L^\infty$-neighbourhood of the reference trajectory and, instead of (ii),  we suppose $\h$  differentiable in a neighbourhood of $(\bar y^0(\bar S),\bar y(\bar S))$  and with continuous derivative at  $(\bar y^0(\bar S),\bar y(\bar S))$.}  
	\end{remark} 
	
	\begin{remark}  {\rm  Hypothesis (iii) on the cone $\C$ is  not at all  restrictive. Indeed,  it can  be recovered by  replacing the single vector fields $g_i$ with suitable linear combinations of $\{g_1,\ldots,g_m\}$ and by considering a corresponding  linear transformation of coordinates in $\cR^m$.}  
	\end{remark} 
	
%
	
	
 We will consider the following distance  $\d$  between extended processes.
\begin{definition}\label{ddeff} 
		Let  $z_i:=(S_i, w^0_i, w_i , y^0_i,y_i,\beta_i)$, $i=1,2$,  be extended processes. We set 
		\begin{equation}\label{dL^1}
				\d\big(z_1,z_2\big):=    |S_1-S_2|+ \|(w^0_2, w_2)-(w^0_1, w_1)   \|_{L^1[0,S_1\wedge S_2]}.
		\end{equation} 
	\end{definition}

\begin{remark}\label{Rdinfty}  {\rm Given  an  extended  process $\bar z:=(\bar S,\bar w^0,\bar w , \bar y^0,\bar  y,\bar\beta)$,  by the  continuity of the input-output map it follows that for any $\varepsilon\in(0,1)$ there exists some $\delta>0$ such that,  for any extended process $z:=(S, w^0, w , y^0,y,\beta)$ such that  $\d\big(z,\bar z\big)<\delta$, one has  
$$
d_\infty\big(z,\bar z\big):=    |S-\bar S|+ \|( y^0,y,\beta)-(\bar y^0,\bar  y,\bar\beta)   \|_{L^\infty(\rr_+)}<\varepsilon,
$$
where  we extend $(y^0,y,\beta)$ and $(\bar y^0,\bar  y,\bar\beta)$     to $\rr_+$ by setting  $(y^0,y,\beta)(s)=(y^0,y,\beta)(S) $ for all $s>S$ and  $(\bar y^0,\bar  y,\bar\beta)(s)=(\bar y^0,\bar  y,\bar\beta)(\bar S)$   for all $s>\bar S$.}
\end{remark} 
	

	Let us  introduce two concepts of local infimum gap, depending on  whether we focus on the extended problem or the embedded strict-sense problem.
	\subsection{Extended processes with local infimum gap  are isolated processes}\label{SubIG}
	 Let us   use $X_\W$  to denote the set  of extended   processes $(S, w^0,w,y^0,y,\beta)$. Furthermore, let $X_{\W_+}\subset  X_\W$
	be the subset of {\it embedded strict-sense processes}, by which we mean those processes with controls in
	$$
	\W_+:=\{(S,w^0,w)\in\W: \ w^0>0 \ \text{a.e.}\}.
	$$
	Let us fix a feasible extended   process $ \hat z:=(\hat S,\hat w^0,\hat w,\hat y^0,\hat y,\hat \beta)$. 
	 For every    $r>0$,  the   sets  $\mathcal{R}^{r}_{\W_+}(\hat z)$, $\mathcal{R}^{r}_\W(\hat z) \subset \cR\times\N\times\cR$,  defined as 
 $$
		\begin{array}{l}
			\mathcal{R}^{r}_{\W_+}(\hat z)  := \Big\{(y^0,y,\beta)(S): \   z=(S, w^0,w,y^0,y,\beta)\in X_{\W_+}, 
			\ \d \left(z, \hat z\right)< r  \Big\},
			\\[1.0ex]
			\mathcal{R}_{\W}^{r}(\hat z):=  \Big\{(y^0,y,\beta)(S): \   z=(S, w^0,w,y^0,y,\beta)\in X_{\W}, 
			\ \d  \left(z, \hat z\right)< r  \Big\},
		\end{array}
		$$ 
	will be  called the   \textit{reachable set} and the \textit{extended  reachable set} (of radius $r$), respectively.
	The occurrence  of  a local infimum gap at an extended process is captured by the following definition: 
	\begin{definition}\label{gap-generator}  Let $\hat z =(\hat S,\hat w^0,\hat w,\hat y^0,\hat y,\hat \beta)$ be a  feasible extended  process. We say that 
		{\rm	there is a local  infimum  gap   at $\hat z$} if  
		there exists $r>0$ such that 
		\bel{igc} 
		\h\big(\hat{y}^0(\hat{S}),\hat{y}(\hat{S}) \big) <\inf_{(y^0,y,\beta)\in \mathcal{R}^{r}_{\W_+}(\hat z) \cap (\cS\times[0,K]) }\h(y^0,y).  
		\eeq 
	\end{definition}
  
 As noted observed in \cite{MRV}, the assumed continuity of the cost implies that  the local infimum  gap condition \eqref{igc} is, despite the name, a fully dynamic topological property, independent of the specific cost function $\Psi$.  Indeed, it just reflects the fact that there are no feasible embedded strict-sense processes in a sufficiently small $\d$-neighborhood  of the  reference  extended process $(\hat S,\hat w^0,\hat w,\hat y^0,\hat y,\hat \beta)$. To be more precise, let us observe that here we are adopting the control distance $\d$, whereas  \cite{MRV}   refers to the $L^\infty$-distance between trajectories. So, to make things rigorous,   let us introduce the notion of {\em isolated process} (with respect to a topology that is different from the one in    \cite{MRV})).
	
	\begin{definition}\label{Diso} A feasible extended    process    $\hat z=(\hat S,\hat w^0,\hat w,\hat \alpha,\hat y^0,\hat y,\hat \beta)$ is called {\rm isolated} if  $ \mathcal{R}^{r}_{\W_+}(\hat z) \cap (\cS\times[0,K])=\emptyset$ for some $r>0$.
    \end{definition} 
	
		 Arguing similarly to  the proof of  \cite[Prop. 2.1]{FM222},  we  get:
	 
	\begin{lma}\label{lemma infgap}   Let ${ \hat z} = (\hat S,\hat w^0,\hat w,\hat \alpha,\hat y^0,\hat y,\hat \beta)$ be a feasible extended   process.   The following   statements  are equivalent:
		\begin{itemize}
			\item [i)] 
		 there is a local  infimum  gap at $\hat z$ (according to Definition \ref{gap-generator});
			\item[ii)] 
			the process $\hat z$ is  isolated;
			\item[iii)]  there exists some $\hat r>0$ such that   $\ds \inf_{{ (y^0,y,\beta)}\in  \mathcal{R}^{r}_{\W_+}{( \hat z)}\cap (\cS\times[0,K])}\Phi(y^0,y) = +\infty$    for  every  continuous function $\Phi$ and every $r\in [0,\hat r]$.
		\end{itemize}
	\end{lma}

\subsection{Embedded strict-sense minimizers with local infimum gap  are limits  of isolated processes}\label{SubIGstrict}

Let us begin with the notions of  embedded strict-sense process and infimum gap related to it. 
\begin{definition}\label{gap-generator2}  Let $ \bar z = (\bar S,\bar w^0,\bar w,\bar y^0,\bar y,\bar \beta)$ be a  feasible  embedded  strict-sense process. We say that $\bar z$ is a {\em local strict-sense minimizer}  if  
	there exists $r>0$ such that   
	\bel{igc2} 
	\h\big(\bar{y}^0(\bar{S}),\bar{y}(\bar{S}) \big) = \min_{(y^0,y,\beta)\in \mathcal{R}^{r}_{\W_+}(\bar z) \cap (\cS\times[0,K]) }\h(y^0,y).  
	\eeq 
	Furthermore,  we say that {\em there is a local infimum gap at the local strict-sense minimizer $\bar z$}   if    $\bar z$ is not a local minimum of the extended problem, namely,  if, for any $\varepsilon >0$, there exists a feasible extended  process   $z_\varepsilon = ( S_\varepsilon ,w^0_\varepsilon, w_\varepsilon,   y^0_\varepsilon, y_\varepsilon,\beta_\varepsilon)$,
	such that
	$$
		\d(z_\varepsilon,\bar z) < \varepsilon  \  \text{ and } \ \h(y^0_\varepsilon(S_\varepsilon),y^\varepsilon(S_\varepsilon) ) < 	\h({\bar y^0}(\bar S),\bar y(\bar S)).
		$$
\end{definition}
The following key result establishes a connection between a local infimum gap at a strict-sense minimizer and isolated extended processes. Unlike Lemma \ref{lemma infgap}, which shows that the presence of a local infimum gap at an extended process implies that the process is isolated, the proposition below states that a local strict-sense minimizer exhibiting a local infimum gap—while not necessarily isolated—is the {\em limit} of a sequence of isolated processes.
	\begin{prop}\label{lemma infgap2}  Consider a local   strict-sense minimizer  $(\bar S,\bar w^0,\bar w,\bar y^0,\bar y,\bar \beta)$  at which there is a local infimum gap (according to Def. \ref{gap-generator2}). Then,  
	there is a sequence  $(\hat S_n,{\hat w^0}_n,\hat w_n ,    {\hat y^0}_n,\hat  y_n,\hat\beta_n)_{n}$  of isolated feasible extended processes such that
		$$
		 \lim_{n\to \infty}\d\big((\hat S_n,{\hat w^0}_n,\hat w_n ,  {\hat y^0}_n,\hat  y_n,\hat\beta_n), (\bar S,\bar w^0,\bar w,\bar y^0,\bar y,\bar \beta)\big) =0.
		$$ 
 { In} particular, we have
		$ \lim_{n\to \infty} ({\hat y^0}_n(S_n),\hat  y_n(S_n)) = (\bar y^0(\bar S),\bar  y(\bar S))$.
\end{prop}

\begin{proof}
By hypothesis,  $\bar z=(\bar S,\bar w^0,\bar w,   \bar y^0,\bar  y,\bar\beta)$ is a local   strict-sense minimizer at which there is a local infimum gap. Therefore, since 
$\bar z$ is a local strict-sense minimizer, there exists $\delta>0$ such that, for every feasible embedded strict-sense process $(S, w^0, w,  y^0, y,\beta)$ satisfying 
\begin{equation}\label{in1hyp}\d\big((S, w^0, w,  y^0, y,\beta),(\bar S,\bar w^0,\bar w,   \bar y^0,\bar  y,\bar\beta)\big)<\delta,\end{equation} one has
\begin{equation}\label{in1}
\h(\bar y^0(\bar S),\bar  y(\bar S)) \leq \h( y^0(S),  y(S)).
\end{equation}
Moreover, because a local infimum gap occurs at $\bar z$, if  $(\varepsilon_n)_n\subset  (0,\delta/2) $ is a sequence decreasing to $0$, for every index  $n$   we can  choose a feasible extended process $\hat z_n=(\hat S_n,{\hat w^0}_n,\hat w_n ,   {\hat y^0}_n,\hat  y_n,\hat\beta_n)$ such that
\begin{equation}\label{limite} 
\h({\hat y^0}_n(\hat S_n),\hat  y_n(\hat S_n)) < \h (\bar y^0(\bar S),\bar  y(\bar S))  \quad \text{and}\quad \d\big(\hat z_n, \bar z\big) < \varepsilon_n .
\end{equation}
Hence, for any feasible embedded strict-sense process $z=(S, w^0, w,  y^0, y,\beta)$ such that 
$
\d\big(z, \hat z_n\big)< \delta/2
$,
one gets 
$$
\displaystyle\d\big(z, \bar z\big)\le  \d\big(z, \hat z_n\big) 
+ \d\big( \hat z_n , \bar z \big)  < \frac\delta 2 + \frac\delta 2 = \delta,
$$
namely \eqref{in1hyp}.
 In  turn,  this  yields \eqref{in1}.
It follows that for each $n$ and every  feasible embedded strict-sense process $(S, w^0, w,  y^0, y,\beta)$ such that $$\d((S, w^0, w,  y^0, y,\beta), (\hat S_n,{\hat w^0}_n,\hat w_n ,{\hat y^0}_n,\hat  y_n,\hat\beta_n)<\displaystyle\frac\delta 2$$  one has 
$$
\h({\hat y^0}_n(S_n),\hat y_n(S_n)) < \h (\bar y^0(\bar S),\bar  y(\bar S)) \le  \h({y^0}(S), y(S)).
$$
Therefore, for every $n$,  the extended process $(\hat S_n,{\hat w^0}_n,\hat w_n ,{\hat y^0}_n,\hat  y_n,\hat\beta_n)$ exhibits a local infimum gap and is thus isolated by Lemma \ref{lemma infgap}. The conclusion follows since the sequence
$(\hat S_n,{\hat w^0}_n,\hat w_n ,{\hat y^0}_n,\hat  y_n,\hat\beta_n)_{n}$ converges, in the $\d$-topology, to the embedded strict-sense process $(\bar S,{\bar w^0},\bar w ,  {\bar y^0},\bar  y,\bar\beta)$ (see \eqref{limite}). The convergence  $\ds\lim_{n\to \infty} ({\hat y^0}_n(S_n),\hat  y_n(S_n)) = (\bar y^0(\bar S),\bar  y(\bar S)) $ is a consequence of the continuity of the input–output map.
\end{proof}

	\subsection{An illustrative example}\label{Sex}

    In this subsection we exhibit a very simple example where both types of local infimum gap investigated in this paper occur. In particular, we identify a strict-sense minimizer  which is not a local minimizer of the extended problem and also an extended minimizer with cost strictly smaller than the infimum cost of the strict-sense problem.

 Let us consider the fixed final constraint  $$\T:= \{1\}\times \left([0,1]\times\{1\}\right),$$ the payoff  
	$$\Psi:\rr^2\to\rr \qquad \Psi(x^1,x^2):=(x^1-1)^2,$$ and  the following minimum control problem with a scalar, unbounded control:
	$$\text{minimize} \  \Psi(x^1(T),x^2(T)) 
	$$
	over controls $u\in L^1([0,T],[0,\infty))$ and trajectories $(x,v)=(x^1,x^2,v)\in AC([0,T],\rr^3)$ satisfying$$
	\left\{
	\begin{array}{l}
		\displaystyle\frac{dx}{dt}(t)=f(x(t)) + g_1(x(t))u(t), \\[1.0ex]
		\displaystyle\frac{dv}{dt}(t)=|u(t)| \\\\
		(x^1,x^2,v)(0)=(0,0,0), \\\\ (T,x^1(T),x^2(T))\in\T  \qquad v(T)\le K:=1,
	\end{array}\right.
	$$
	where  $$f(x):=\begin{pmatrix}0\\1\end{pmatrix},\qquad g_1(x):=\begin{pmatrix}1\\\eta(x^2)\end{pmatrix}, \,\,\,\,  \,\,\,\eta(x^2):=\frac12\sin(\pi( x^2-1/2))-\frac12.
	$$
	Notice that $T\equiv 1$, for any feasable process $(u,x)(\cdot)$.
	It is immediate to see that   the curve   $(\bar x^1,\bar x^2,\bar v)(t) = (0,t,0)$  is the only  trajectory of the  system which is also feasible, namely it is complying with the final target.  This trajectory is obtained by implementing the control $\bar u\equiv 0$. The  {final position is $(\bar x^1,\bar x^2)(1)=(0,1)$,} so that the corresponding optimal cost is $ \Psi(0,1)= 1$.
	
	The impulsive extension is given by
	$$\text{minimize} \  \Psi(y^1(S),y^2(S)) $$
	over $S>0$, controls $(w^0,w)\in L^1([0,S],W)$, 
		and trajectories 
		
		\noindent $(y^1,y^2,\beta)\in AC([0,S],\rr^3)$ satisfying 
	$$
	\left\{\begin{array}{l}
		\displaystyle\frac{dy^0}{ds}(s)=w^0(s), \\[1.0ex]
		\displaystyle\frac{dy}{ds}(s)=f(y(s))w^0(s) + g_1(y(s))w(s)\\[1.0ex]
		\displaystyle\frac{d\beta}{ds}(s)=|w(s)|, \quad\text{a.e. $s\in[0,S]$,} \\[1.0ex]
		(y^0,y^1,y^2,\beta)(0)=(0,0,0, 0), \\\\   y^0(S)=1, \quad (y^1, y^2)(S)\in \T, \quad \beta(S)\le 1,
	\end{array}\right.
	$$
	where $W:=\{(w^0,w)\in\rr^2: \ w^0, w\ge0, \ w^0+|w|=1\}$. The embedded strict-sense process corresponding to the original process $(1,\bar u, \bar x^1,\bar x^2,\bar v)$ is  
	
	\noindent$\bar z:=(\bar S, \bar w^0,\bar w, \bar y^0,\bar y^1,\bar y^2,\bar\beta),$ where 
	$$
	\bar S=1, \quad (\bar w^0,\bar w)\equiv (1,0), \quad
	(\bar y^0,\bar y^1,\bar y^2,\bar\beta)(s)=(s, 0,s,0) \ \  \text{for all $s\in[0,1]$,}	
	$$
	Hence, $\bar z$   is a strict-sense minimizer for the embedded strict sense problem, with cost $\Psi(\bar y^1(\bar S),\bar y^2(\bar S))=1$.  We claim that there is a local infimum gap at   the strict-sense minimizer $\bar z$. Indeed, for any $r\in(0,1]$, the extended control $(S_r, w^0_r,w_r)$ defined as  
	$$
	S_r:=1+r, \quad (w^0_r,w_r)(s):=\begin{cases} (1,0) &\quad\text{if $s\in[0,1]$,}\\
		(0,1) &\quad\text{if $s\in (1,1+r]$}\end{cases}
	$$
	gives rise to the feasible extended trajectory
	$$
	(y^0_r,y^1_r,y^2_r,\beta_r)(s)=\begin{cases} (s, 0,s,0) &\quad\text{if $s\in[0,1]$,}\\
		(1,s-1,1,s-1) &\quad\text{if $s\in(1,1+r]$.}\end{cases}
	$$ Considering the corresponding processes 
	$z_r:=(S_r, w^0_r, w_r, y^0_r,y^1_r, \beta_r)$, $r\in[0,1]$, one has  
	$$ \d\big(z_r,\bar z\big) =|\bar S-S_r|=r\qquad \Psi(y^1_r(S_r),y^2_r(S_r))=1-r<1=\Psi(\bar y^1(\bar S),\bar y^2(\bar S)),$$ 
	which shows that there is an infimum gap at $\bar z$.
	
%

\section{Infimum gaps and abnormal minimizers}\label{mainsection}

 In this section, we will present the main results concerning the relation between infimum gaps and abnormality.  In particular, we will show that the presence of a local infimum gap for both an extended process and a strict-sense minimizer implies higher-order abnormality. 
	\subsection{Iterated Lie brackets}\label{sps}   In order to give the notion of higher-order  extremal, let us  recall some basic facts concerning iterated  Lie brackets.
		If $h_1$, $h_2$ are $C^1$ vector fields on $\rr^n$,
		 the {\it Lie bracket of $h_1$ and $h_2$} is defined (on any local system of coordinates),
		  as
	$$
x\mapsto[h_1,h_2](x) := Dh_2(x)\cdot h_1(x) -  D h_1(x)\cdot h_2(x) (= - [h_2,h_1](x)).
	$$ 
	As is well-known, the map $[h_1,h_2]$ is a true
	vector field, i.e. it can be defined intrinsically. If the vector fields are sufficiently regular  one can iterate the bracketing process: for instance, given a $4$-tuple ${\bf h}:=(h_1,h_2,h_3,h_4)$ of vector fields  one can construct  the brackets $[[h_1,h_2],h_3]$,  $[[h_1,h_2],[h_3,h_4]]$,  $[[[h_1,h_2],h_3],h_4]$, $[[h_2,h_3],h_4]$.   Accordingly, one can consider the  {\it  (iterated)  formal brackets } $B_1:=[[X_1,X_2],X_3]$, $B_2:=[[X_1,X_2],[X_3,X_4]]$, $B_3:=[[[X_1,X_2],X_3],X_4]$, $B_4:=[[X_2,X_3],X_4
	]$ (regarded as  sequence of letters $X_1,\ldots,X_4$, commas, and left and right square  parentheses), so that, with obvious meaning of the notation, one has $B_1({\bf h}) = [[h_1,h_2],h_3]$,   $B_2({\bf h}) = [[h_1,h_2],[h_3,h_2]]$, $B_3({\bf h}) =[[[h_1,h_2],h_3],h_4]$, $B_4({\bf h}) =[[h_2,h_3],h_4]$. 
	
	The {\it length} of a  formal bracket is   the number of letters that are   involved in it.  For instance, the above  brackets $B_1, B_2, B_3, B_4$ have lengths equal to  $3$, $4$, $4$, and $3$, respectively.  By convention, we declare that a single  variable $X_i$  is a formal bracket of length $1$.
	
	The  {\it switch-number}  of a (formal)  bracket  ${B}$ is the number $r_{_{B}}$ defined
	recursively on the nested structure of the bracket as (see \cite{FeleqiRampazzo2017}):  
	\vsmm ${r_{_{B}}} := 1,$ if   ${B}$  has length  $1$;  
	$ r_{_{B}}:= 2\big(r_{_{B_1}}+r_{_{B_2}}\big)$ if  ${B}=[B_1,B_2]$. 
	\vsmm
	\noindent For instance, the switch-numbers of $[[X_3,X_4],[[X_5,X_6],X_7]]$ and $[[X_5,X_6],X_7]$ are $28$ and $10$, respectively. If there is no risk of confusion, we sometimes refer informally to the {\it length and switch-number of Lie brackets of vector fields}. \footnote{Instead of speaking of  length and switch-number {\it of formal Lie brackets.} }
	
	The regularity of a (non-formal) iterated Lie  bracket  depends on both the nested structure of the underlying formal brackets and   the involved  vector fields. We will  use the following notion of {\it bracket regularity} for a string of vector fields (for a more  rigorous  definition we refer to \cite{FeleqiRampazzo2017}): 
	\begin{definition}[Bracket regularity]   Fix $k\in\cN$. If $\mu\geq 0$,  $r\ge1$, and $\nu\geq \mu+r$ are integers,  $B = B(X_{\mu+1},\ldots,X_{\mu+r})$ is an iterated  formal bracket, and   ${\bf h}=(h_1,\ldots,h_\nu )$  is a string of vector fields, we say that {\rm  ${\bf h}$ is of class $C^{B+k}$} if there is a $\nu$-tuple $(j_1,\ldots,j_\nu)\in\cN^\nu$  such that
		$h_i$ is of class $C^{j_i}$ for any $i=1,\ldots,\nu$  and 
		$B({\bf h})$ is a vector field of class $C^k$. In this case, we   call $(B,\mathbf{h})$   an {\rm admissible $C^k$ bracket pair.} 
	\end{definition}
	For instance, if   $B=$ $[ [[X_3,X_4],[X_5,X_6]],X_7 ]$,   for any $k\geq 0$, a string  ${\mathbf h}=$ \linebreak
	$(h_1,h_2,h_3,h_4,h_5,h_6,h_7,h_8)$ is of class $C^{B+k}$ provided the vector fields  $h_3,h_4,h_5,h_6$ are of class  $C^{3+k}$ and $h_7$ is of class $C^{1+k}$.


	\subsection{Higher-order extremals}\label{Ss1} 
	Let us  set 
	\bel{DS+}
	\CC:=\{(w^0,w)\in\cR_+ \times \C: \,  w^0+|w|=1\}
	\eeq
	and let us  consider the  {\it unmaximized Hamiltonian} $H: (\rr^n)^2 \times\cR\times\cR\times\cR_+\times\C   \to \cR$\footnote{More precisely, the domain of $H$ is $\rr^n\times(\rr^n)^*\times\rr^*\times\rr^*\times\cR_+^*\times\C$. } $$
	H(x,p,p_0,\pi,w^0,w ):= p_0w^0 + p\cdot\Big(f (x) w^0 +  \sum_{i=1}^{m}  g_{i}(x) w^{i}\Big) + \pi | w|.
	$$  
	
	To give the notion of higher-order extremal, we need one more definition: 
	
	\begin{definition} For every integer $k\geq 0$, 
		we will  use $\B^k$ to denote the (possibly empty)  set of  
		admissible $C^k$  bracket pairs  $(B,\mathbf{h})$,  such that  $\mathbf{h}:=(h_1,\ldots,h_\nu)$   is a $\nu$-tuple of vector fields verifying  $\{h_1,\ldots,h_\nu\}\subseteq \{g_1,\ldots,g_{m_1}\}$.
	\end{definition} 
	Let us consider, for instance, $m_1\ge10$ and the pair 
	$B=[[X_3,[X_4,X_5]],X_6]$, $\mathbf{h} =(h_1,\ldots,h_6):=(g_8,g_{10},g_1,g_4,g_3,g_1)$.
	Then, the vector field $B(\mathbf{h})$ coincides with the iterated Lie bracket $[[g_1,[g_4,g_3]],g_1]$. Moreover,  for every $k\in \mathbb{N}$,  the pair
	$(B,\mathbf{h})$ belongs to  $\in \B^k$ (in particular, $B(\mathbf{h})\in C^k$) provided $g_1\in C^{2+k}$, and $g_{3},g_4\in C^{3+k}$.

	\begin{definition}[Higher-order extremal]\label{HOExtremal} Let $(\bar S,\bar w^0,\bar w , \bar y^0,\bar  y,\bar\beta)$
		be a  feasible extended  process. Let  $\K$  be a QDQ   approximating cone to  the target $\cS$ at $(\bar y^0(\bar S),\bar y(\bar S))$. 
		We say that  $(\bar S,\bar w^0,\bar w , \bar y^0,\bar  y,\bar\beta)$ is  a {\em higher-order  $\h$-extremal} with respect to $\K$  if  there exist  a adjoint map $p\in AC([0,\bar S], \rr^n)$ and multipliers   $(p_0,\pi,\lambda)\in \cR\times \cR_-
		\times \cR_+$  such that following conditions {\rm(i)-(vi)} below are satisfied:
		\begin{itemize}
			\item[{\rm (i)}]  {\sc (non-triviality)} the triple $(p_0, p , \lambda)$ is non trivial, i.e.
			\begin{equation} 
				\label{fe1}
				(p_0, p , \lambda) \not= (0, 0,0) \,;
			\end{equation}
			furthermore, if 
			the trajectory $\bar y $ is not purely impulsive, namely, if 
			$ \bar y^0(\bar S)>0$, then \eqref{fe1} can be strengthened  to 
			\begin{equation}\label{strongfe1}
				(p  , \lambda) \not= (0,0); 
			\end{equation}
			\item[{\rm (ii)}] {\sc (non-tranversality)}  
			\begin{equation}
				\label{fe4}
				(p_0,p(\bar S),\pi) \in \left[-\lambda D\Psi  (\bar y^0(\bar S),\bar y(\bar S))- \K^\perp\right]\times J_K,
			\end{equation}
			where $J_K:=\{0\}$ if $\bar\beta(\bar S) < K $, and $J_K:=\cR_+$ 
			if $\bar\beta(\bar S) = K$; in particular,   
			\bel{piestzero}\pi = 0 \quad\hbox{provided}\quad \bar\beta(\bar S)< K  ; \eeq 
			\item[{\rm (iii)}] {\sc (Hamiltonian equations)}  the path $ (\bar y, p) $ satisfies 
			\begin{equation}
				\label{fe2def}
				\displaystyle  \frac{d}{ds}(\bar y,p) (s)\,=\, {\bf X}_{\bar H}\left(s,\bar y(s), p(s)\right) \quad\text{for a.e. $s\in [0,\bar S]$,}
			\end{equation} 
			where $\bar H=\bar H(s,y,p):= H\big(y,p, p_0,\pi,\bar w^0(s),\bar w(s)\big)$, and ${\bf X}_{\bar H}$ denotes   the ($s$-dependent) Hamiltonian vector field corresponding to ${\bar H}$;\footnote{If ${\mathcal H}(s,y,p)$ is a differentiable map on the cotangent bundle $T^*\cR^n$, in any local system of canonical coordinates $(y,p)$ the Hamiltonian vector field ${\bf X}_{\mathcal H}$ corresponding to ${\mathcal H}$ is defined as    $\ds {\bf X}_{\mathcal H}(s,y,p) := \left(\frac{\partial\mathcal H}{\partial p}, - \frac{\partial\mathcal H}{\partial y}\right)(s,y,p)$, so that \eqref{fe2def} coincides with the extended system coupled with the  usual adjoint equation.}
			\item[{\em (iv)}] {\sc (First order maximization)} For a.e. $s\in [0,\bar S]$, 
			\begin{equation}\label{fe3}
				\begin{array}{lr}
				\displaystyle	\bar H(s,\bar{y}(s),p(s))=  \max_{(w^0,w )\in \CC }  H(\bar y(s), p(s), p_0 ,\pi,w^0,w ),
				\end{array}
			\end{equation}
			 {which, in particular, yields}
			\begin{gather}
				p(s)\cdot g_i(\bar y(s))=0,   \qquad \text{for all } s\in[0,\bar S],\   i=1,\dots,m_1,  \label{pg0hi1}
			\end{gather}
	{	as soon as $\bar\beta(S)<K$;}
			\item[{\em (v)}] {\sc (Vanishing of the Hamiltonian)} 
			\bel{engine}
			\max_{(w^0,w  )\in \CC  }  H(\bar y(s), p(s), p_0 ,\pi,w^0,w )=0, \quad \text{for all }s\in[0,\bar S]
			\eeq
			\item[{\em (vi)}] {\sc (Higher-order conditions)} if  $\bar\beta(S)<K$ and  $(B,\mathbf{h})\in \B^0$,  
			\begin{gather}
				p(s)\cdot  B(\mathbf{h})(\bar y(s))=0,   \quad \text{for all } s\in [0,\bar S] ; \label{pg000hi}
			\end{gather}
			furthermore, if  $(B,\mathbf{h})\in \B^1$    for  a.e.   $s\in [0,S]$, one has\footnote{When $m=m_1=1$ equality \eqref{MP111new} reduces to the Legendre-Clebsch-type  condition 
				$$  p(s)\,\cdot\, \big[{f }, g \big](\bar y(s))\,\bar w^0(s) =0 .$$
			 } 
			\bel{MP111new}
			p(s)\,\cdot\, \bigg(\big[{f  }, B(\mathbf{h})\big](\bar y(s))\,\bar w^0(s) + \ds\sum_{j=m_1+1}^{m}
			\big[g_j,B(\mathbf{h})\big](\bar y(s))\,\bar w^j(s)\bigg) =0.
			\eeq 
		\end{itemize}
	\end{definition}

	\begin{definition}\label{HONormal} 
		Let  $(\bar S,\bar w^0,\bar w , \bar\alpha,  \bar y^0,\bar  y,\bar\beta)$ be  a   feasible extended  process, which, 
 for a given   QDQ   approximating cone  $\K$  to the target $\cS$ at $(\bar y^0(\bar S),\bar y(\bar S))$,
		is  a higher-order  $\h$-extremal wirth respect to $\K$.  We say  that $(\bar S,\bar w^0,\bar w , \bar\alpha,  \bar y^0,\bar  y,\bar\beta)$ is a {\em normal higher-order   $\h$-extremal with  respect to $\K$}  if,  for any  choice  of the multipliers  $(p_0, p , \pi,\lambda)$,   one has $\lambda \ne0$. Otherwise, we say that it  is an  {\em abnormal higher-order   extremal (with  respect to $\K$)}.\footnote{The fact that a higher-order   extremal, as well as  a classical  extremal, is abnormal (with  respect to  some $\K$) does not depend on the cost function $\h$.}
	\end{definition}
	 \begin{remark}{\rm The notion of extremality  depends on the QDQ approximating cone 
$\K$.  This one used in this paper is slightly more general than the one used in the higher-order Maximum Principle established in \cite{AMR20}, since we work with QDQ approximating cones, which generalize the Boltyanski approximating cones  {used} in \cite{AMR20} (see also \cite{AMRCDC}).}\end{remark}

	\subsection{Local infimum gaps and higher-order abnormality}  Let us introduce our main results. 
	\begin{thm}[Local infimum gap at an extended process]\label{ThIsolated}
		Consider a   feasible extended    process $(\bar S,\bar w^0,\bar w ,   \bar y^0,\bar  y,\bar\beta)$  at which  there is  a local  infimum gap.   Then,  for every QDQ   approximating cone  $\K$  to the target $\cS$ at $(\bar y^0(\bar S),\bar y(\bar S))$,  the process $(\bar S,\bar w^0,\bar w ,   \bar y^0,\bar  y,\bar\beta)$ is an abnormal  higher-order   extremal with  respect to $\K$.  
	\end{thm}
	This result holds for  extended   processes  at which  there is  a local  infimum gap with respect to the $L^\infty$-distance between trajectories and was proven in  \cite{MPR} using a set separation approach.  We refer to \cite{MPRArXiv} for the adaptation of the proof of  \cite[Theorem 1]{MPR}  to the present case,   which involves the  control distance $\d$ defined in \eqref{ddeff}.

	Equivalently to  Theorem \ref{ThIsolated}, one has  the following sufficient condition for the absence of  a local infimum gap  at an extended process: 
	\begin{thm}[Higher-order normality and no-gap at an extended process]\label{Thnormality} Suppose that 
		$\hat z=(\hat S,\hat w^0,\hat w ,  \hat y^0,\hat  y,\hat\beta)$ is a   feasible extended   process  which satisfies, for some $r>0$,
		$$
		\h(\hat y^0(\hat S),\hat y(\hat S))\le\inf_{(y^0,y,\beta)\in \mathcal{R}^{r}_{\W_+}(\hat z) \cap (\cS\times[0,K])} \h(y^0,y) .
		$$
		If $(\hat S,\hat w^0,\hat w ,  \hat y^0,\hat  y,\hat\beta)$ is  a normal higher-order  $\h$-extremal
	with respect to some   QDQ approximating cone $\K$ to   $\T$ at $(\bar y^0(\bar S),\bar y(\bar S))$, then 
		 there is no local infimum gap at $(\hat S,\hat w^0,\hat w ,  \hat y^0,\hat  y,\hat\beta)$.  \end{thm}
	
 The  main result of this second part of the  paper consists in  following theorem: 
 \begin{thm}[Local infimum gap at a  strict-sense minimizer]\label{Thstrict_sense}  Suppose that  $(\bar S,\bar w^0,\bar w ,   \bar y^0,\bar  y,\bar\beta)$ is  a  strict-sense minimizer at which  there is  a local  infimum gap, and assume  that the target $\T$ is Quasi Prox-Regular at $(\bar y^0(\bar S),\bar y(\bar S))$.  Then, as soon as we choose  the Clarke tangent cone $\K:=T^C_\T(\bar y^0(\bar S),\bar y(\bar S))$ as a  QDQ approximating cone  to the target $\T$ at $(\bar y^0(\bar S),\bar y(\bar S))$,  the process $(\bar S,\bar w^0,\bar w ,   \bar y^0,\bar  y,\bar\beta)$ is an abnormal  higher-order   extremal with respect to $\K$.\footnote{In view of Theorem \ref{Th_T^C},  the Clarke tangent cone $T^C_\T((\bar y^0(\bar S),\bar y(\bar S))$ is effectively a   QDQ approximating cone  to the target $\T$ at $(\bar y^0(\bar S),\bar y(\bar S))$}
\end{thm}	
    \begin{proof} 
  Since $\bar z:=(\bar S,\bar w^0,\bar w ,   \bar y^0,\bar  y,\bar\beta)$ is  a  strict-sense minimizer at which  there is  a local  infimum gap, it follows from Proposition \ref{lemma infgap2}  that there exist  a decreasing sequence $(\rho_j)_{j}$, with $0<\rho_j<1$,  converging to 0 and a sequence of isolated, feasible   extended processes $z_j:=(S_j,w^0_j,w_j,y^0_j,y_j, \beta_j)$ satisfying
\bel{form}
  \d(z_j,\bar z)=|S_j-\bar S|+\int_0^{S_j\land \bar S}[|w^0_j(s)-\bar w(s)|+|w^0_j(s)-\bar w(s)|]\,ds<\rho_j \quad\text{ for any $j\ge 1$.}
    \eeq

  {\em Step 1.}  Thanks to Theorem \ref{Th_T^C},  Lemma \ref{lemma infgap}, and Theorem \ref{ThIsolated}, if for any $j\in\cN$ we use the Clarke tangent cone as  QDQ approximating cone to the target at $(y^0_j(S_j),y_j(S_j))$,  we obtain that there exist a lift $(y_j, p_j)\in AC([0, S_j],(\rr^n)^2)$ and multipliers   $(p_{0_j},\pi_j,\lambda_j=0)\in \cR\times \cR_-
		\times \cR_+$  such that the following conditions are satisfied: 
			\begin{equation} 
				\label{pfe1}
				|p_{0_j}|+\| p_j\|_{L^\infty(0, S_j) }\ne 0 \quad\text{ ($\| p_j\|_{L^\infty(0, S_j)} \ne0$ if 	$y^0_j(S_j)>0$)	   ; }
			\end{equation}	
			 \begin{equation}
				\label{pfe4}
				(p_{0_j},p_j(S_j),\pi_j) \in   - N^C_\T(y_j^0(S_j),y_j(S_j))\times J_K,
			\end{equation}
			where, setting $\beta_j(s)\colon=\int_0^s|\omega_j(\sigma)| d\sigma$, we take $J_K:=\{0\}$ if $\beta_j(S_j) < K $, and $J_K:=[0,+\infty)$ 
			if $\beta_j(S_j) = K$. In particular,   
			\bel{ppiestzero}\pi_j = 0 \quad\hbox{provided}\quad \beta_j(S_j)< K   ; \eeq 
					\item 	\begin{equation}
				\label{fe2}
				\displaystyle  \frac{d}{ds}(y_j,p_j) (s)\,=\, {\bf X}_{H_j}\left(s, y_j(s), p_j(s)\right) \quad\text{for a.e. $s\in [0,S_j]$,}
			\end{equation} 
			where $H_j= H_j(s,y,p):= H\big(y,p, p_{0_j},\pi_j, w^0_j(s), w_j(s)\big)$, and ${\bf X}_{H_j}$ denotes   the ($s$-dependent) Hamiltonian vector field corresponding to ${H_j}$;
		\begin{equation}\label{pfe3}
				\begin{array}{l}
					\displaystyle H_j\Big(s, y_j(s), p_j(s) \Big)= \max_{(w^0,w )\in \CC }  H(y_j(s), p_j(s), p_{0_j} ,\pi_j,w^0,w ), \quad\text{for a.e. $s\in [0,\bar S]$,}
		\end{array}
			\end{equation}
			and, as soon as $\beta_j(S)<K$     ;
		 		\begin{gather}
			p_j(s)\cdot g_i( y_j(s))=0,   \qquad \text{for all } s\in[0, S_j],\   i=1,\dots,m_1  ;  \label{ppg0hi1}
			\end{gather} 
		  	\bel{pengine}
			\max_{(w^0,w  )\in \CC  }  H(y_j(s), p_j(s), p_{0_j} ,\pi_j,w^0,w )=0, \quad \text{for all }s\in[0, S_j]    ;
			\eeq
			\item   if  $\beta_j(S_j)<K$ and  $(B,\mathbf{h})\in \B^0$,  
			\begin{gather}
				p_j(s)\cdot  B(\mathbf{h})( y_j(s))=0,   \quad \text{for all } s\in [0, S_j]   ; \label{ppg000hi}
			\end{gather}
		 furthermore, if  $(B,\mathbf{h})\in \B^1$,  for  a.e.   $s\in [0,S_j]$ one has 
			\bel{pMP111new}
			p_j(s)\,\cdot\, \bigg(\big[{f  }, B(\mathbf{h})\big](y_j(s))\, w_j^0(s) + \ds\sum_{i=m_1+1}^{m}
			\big[g_i,B(\mathbf{h})\big]( y_j(s))\,w_j^i(s)\bigg) =0.
			\eeq

According to Remark \ref{Rdinfty} and \eqref{form},  $\lim_{j\to\infty}d_\infty(z_j,\bar z) = 0$. In particular, the (constant and continuous extensions to $\mathbb{R}_+$ of) $(y^0_j, y_j, \beta_j)$  converge uniformly to (the constant and continuous extension to $\mathbb{R}_+$ of) $(\bar{y}^0, \bar{y}, \bar{\beta})$ in $\mathbb{R}_+$. Moreover, we have
\[
\lim_{j \to \infty} (y^0_j, y_j, \beta_j)(S_j) = (\bar{y}^0, \bar{y}, \bar{\beta})(\bar{S}).
\]

  {\em Step 2.}  Under the regularity and boundedness assumptions we have adopted, the Hamiltonian system associated with each $(y_j, p_j)$ becomes the usual system 
\begin{equation}\label{adjeq}
 \left\{
\begin{split}
\frac{dy_j}{ds} (s) & = f( y_j(s) )w_j^0(s)+ \sum_{i=1}^{m}g_{i}( y_j(s))w_j^{i}(s), \\
\frac{dp_j}{ds} (s) & =-p_j(s)\cdot \left[\frac{\partial f}{\partial x}( y_j(s) )w_j^0(s)+ \sum_{i=1}^{m}\frac{\partial g_{i}}{\partial x}( y_j(s))w_j^{i}(s)\right], \ \ \text{for a.e. $s\in[0,S_j]$,}
\end{split}
\right.
\end{equation}
with the boundary conditions
$$
y_j(0) = \xbo, \quad (p_{0_j},p_j(S_j),\pi_j) \in -N^C_\T(y^0_j(S_j),y_j(S_j))\times J_K.
$$  
By normalizing the multipliers, for each \( j \in \mathbb{N} \), we can assume that
\begin{equation}\label{ntriN}
|p_{0_j}| + \|p_j\|_{L^\infty(0, S_j)} + |\pi_j| = 1.
\end{equation}
Therefore, the sequences \( (p_{0_j})_j \) and \( (\pi_j)_j \) are bounded in \( \mathbb{R} \). Moreover, due to the linearity of the adjoint equation, the sequence  of (continuous, constant extension to \( \mathbb{R}^+ \) of) the functions \( (p_j)_j \) is equibounded and equi-Lipschitz continuous (on \( \mathbb{R}^+ \)). Hence, by the Ascoli–Arzelà Theorem, possibly up to extracting a subsequence  which we still denote by \( (p_{0,j}, p_j, \pi_j)_j \), there exist
\[
p_0 = \lim_{j \to \infty} p_{0_j}\in\cR,\quad \pi = \lim_{j \to \infty} \pi_j \leq 0,
\]
and \( p \in AC(\mathbb{R}_+, \mathbb{R}^n) \) such that
\[
\|p_j - p\|_{L^\infty(\mathbb{R}_+)} \to 0, \quad p_j(S_j)\to p(\bar S) \quad \text{as } j \to \infty.
\]
 At this point, standard arguments (see, for instance, \cite{MRV}) imply that $p$ solves the adjoint equation
$$
\frac{dp}{ds} (s)  =-p(s)\cdot \left[\frac{\partial f}{\partial x}( \bar y(s) )\bar w ^0(s)+ \sum_{i=1}^{m}\frac{\partial g_{i}}{\partial x}( \bar y(s))\bar w^{i}(s)\right], \ \ \text{for a.e. $s\in[0,\bar S]$.}
$$
The transversality condition follows immediately from the fact that the Clarke normal cone coincides with the convex hull of the limiting normal cone; consequently, the condition $(p_{0_j},p_j(S_j),\pi_j)\in -N^C_\T(y^0_j(S_j),y_j(S_j))\times J_K$ is preserved at the limit, so
$$(p_{0},p(\bar S), \pi)\in -N^C_\T(\bar y^0(\bar S),\bar y(\bar S))\times J_K.$$
 Moreover, the maximality condition  \eqref{fe3} is also easily obtained by passing to the limit as \( j \to + \infty \) in \eqref{pfe3}. 

  {\em Step 3.}  It remains to prove the nontriviality condition and, if \( \bar{\beta}(\bar{S}) < K \), also the higher-order conditions.
  
  Suppose first $\bar{\beta}(\bar{S}) = \int_0^{\bar S}|\bar\omega(s)| ds < K$.  Then also $\beta_j(S_j) < K$, which implies that 
 $\pi_j = 0$, { for  any $j$ sufficiently large}. Hence, from \eqref{ntriN} it follows that $p_{0_j} + \|p_j\|_{L^\infty(0, S_j)} = 1$. So, in the limit we get  $\pi=0$ and the nontriviality condition $p_{0} + \|p\|_{L^\infty(0, \bar S)} = 1\ne0$. If, in addition,  $\bar{y}^0(\bar{S}) > 0$, the last nontriviality condition reduces to $\|p\|_{L^\infty(0, \bar S)}\ne0$. Indeed, if  it were  $p\equiv0$, then we would obtain $p_0=1$, $\pi=0$, and, integrating the maximality condition  \eqref{fe3}, we  would get the  contradiction
 $$
0< \bar{y}^0(\bar{S})=\int_0^{\bar S}\bar w^0(s)\,ds=0.
$$ 
 Therefore, $(p_0,p,\pi=0)$ satisfies the nontriviality condition. Since  $\beta_j(S_j) < K$ for  $j$ sufficiently large, \eqref{ppg0hi1} and the higher-order conditions \eqref{ppg000hi}, \eqref{pMP111new} are valid for such $j$. Taking the limit as $j\to+\infty$, these conditions imply the analogous properties \eqref{pg0hi1}, \eqref{ppg000hi}, and \eqref{pMP111new} for $p$.
 
Assume now that  $\bar{\beta}(\bar{S})= K$. In this case, passing to the limit in  \eqref{ntriN}, even if we know that $|p_{0_j}| + \|p_j\|_{L^\infty(0, S_j)}>0$ for all $j$,  in principle we might have $|p_{0}| + \|p\|_{L^\infty(0, \bar S)}=0$, i.e. $p_0=0$, $p\equiv 0$, and $\pi=-1$. However, integrating the maximality condition  \eqref{fe3} we would get the contradiction 
$$
0< K=\bar{\beta}(\bar{S})=\int_0^{\bar S}|\bar w(s)|\,ds=0.
$$ 
Hence, $|p_{0}| + \|p\|_{L^\infty(0, \bar S)}\ne0$ and the nontriviality condition \eqref{fe1} is satisfied.

Now let us show that \eqref{strongfe1} holds true provided  $\bar y^{0}(\bar S)>0$. First of all, observe that  $\lambda=\lim_j \lambda_j =0$, since  $\lambda_j=0$ for every $j\in\mathbb{N}$. Therefore, we    must show that  $\|p\|_{L^\infty(0, \bar S)}\ne 0$.  If $p\equiv0$, we would be able to deduce from  \eqref{fe3} that 
\begin{equation}\label{uffa}
p_0\bar w^0(s) + \pi |\bar w(s)| = \underset{(w^0,w) \in \CC}{\max}
\bigg\{ 
p_0w^0 + \pi |w|
\bigg\}\,=  0\, \mbox{ a.e. }s \in [0,\bar S].
\end{equation}
If $\pi<0$, it would follow from this relation that $p_0 = 0$. This cannot be true since $(p_0,p)=(0,0)$ is in contradiction with  \eqref{fe1}. If, on the other hand, $\pi=0$, it would follow from  \eqref{uffa} that  $p_0<0 $.  But then we would have $\bar w^0(s) =0$  
a.e.,
which would imply
$\bar y^0(\bar S)=0,
$
in contradiction with the hypothesis that   $\bar y^0(\bar S)>0$.  
 
By combining the previous results, we have established the existence of a set of multipliers \((p_0, p, \pi, \lambda = 0)\) for which the strict-sense minimizer \((\bar S,\bar w^0,\bar w ,   \bar y^0,\bar  y,\bar\beta)\) is an abnormal higher-order extremal according to Definition \ref{HONormal}.
  \end{proof}

\end{document}